\documentclass{article}

\usepackage{adjustbox}
\usepackage{amssymb}
\usepackage{amsmath}
\usepackage{amsthm}
\usepackage{appendix}
\usepackage{array}
\usepackage{bm}
\usepackage{booktabs}
\usepackage{caption}
\usepackage{colortbl}
\usepackage{epstopdf}
\usepackage{enumitem}
\usepackage{graphicx}
\usepackage[colorlinks=true, linkcolor=Brown, citecolor=DarkCyan, urlcolor=Olive]{hyperref}
\usepackage{lipsum}
\usepackage{mathtools}
\usepackage{multirow}
\usepackage[sort&compress, square, numbers]{natbib}
\usepackage{orcidlink}        
\usepackage[flushleft]{threeparttable}
\usepackage{url}
\usepackage[svgnames]{xcolor}

\newtheorem{assumption}{Assumption}

\newtheorem{definition}{Definition}

\newtheorem{lemma}{Lemma}

\newtheorem{remark}{Remark}

\newtheorem{theorem}{Theorem}

\DeclarePairedDelimiter\abs{\lvert}{\rvert}%
\DeclarePairedDelimiter\norm{\lVert}{\rVert}%
\makeatletter
\let\oldabs\abs
\def\abs{\@ifstar{\oldabs}{\oldabs*}}
\let\oldnorm\norm
\def\norm{\@ifstar{\oldnorm}{\oldnorm*}}
\makeatother

\newcommand{\tnorm}[1]{{\left\vert\kern-0.25ex\left\vert\kern-0.25ex\left\vert #1\right\vert\kern-0.25ex\right\vert\kern-0.25ex\right\vert}}
\newcommand{\tnormstar}[1]{\tnorm{#1}_*}

\DeclareRobustCommand{\hatj}{{\hat\jmath}}
\newcommand{\checkj}{{\check\jmath}}

\newcommand{\highToIntfPrimal}{\mathcal{P}_{\hatj}^{j}}
\newcommand{\lowToIntfPrimal}{\mathcal{P}_{\checkj}^{j}}

\newcommand{\intfToHighDual}{\mathcal{D}_{j}^{\hatj}}
\newcommand{\intfToLowDual}{\mathcal{D}_{j}^{\checkj}}

\newcommand{\highToIntfPrimalDiscrete}{\widetilde{\mathcal{P}}_{\hatj}^{j}}
\newcommand{\lowToIntfPrimalDiscrete}{\widetilde{\mathcal{P}}_{\checkj}^{j}}

\newcommand{\highToIntfDualDiscrete}{\widetilde{\mathcal{D}}_{\hatj}^{j}}
\newcommand{\lowToIntfDualDiscrete}{\widetilde{\mathcal{D}}_{\checkj}^{j}}
\newcommand{\intfToHighDualDiscrete}{\widetilde{\mathcal{D}}_{j}^{\hatj}}
\newcommand{\intfToLowDualDiscrete}{\widetilde{\mathcal{D}}_{j}^{\checkj}}

\newcommand{\highToIntfPrimalToAux}{\widetilde{\mathcal{P}}_{\hatj}^{j,\hatj}}
\newcommand{\highToIntfPrimalFromAux}{\widetilde{\mathcal{P}}_{j, \hatj}^{j}}
\newcommand{\lowToIntfPrimalToAux}{\widetilde{\mathcal{P}}_{\checkj}^{j,\checkj}}
\newcommand{\lowToIntfPrimalFromAux}{\widetilde{\mathcal{P}}_{j, \checkj}^{j}}

\newcommand{\highToIntfDualToAux}{\widetilde{\mathcal{D}}_{\hatj}^{j,\hatj}}
\newcommand{\highToIntfDualFromAux}{\widetilde{\mathcal{D}}_{j, \hatj}^{j}}
\newcommand{\lowToIntfDualToAux}{\widetilde{\mathcal{D}}_{\checkj}^{j, \checkj}}
\newcommand{\lowToIntfDualFromAux}{\widetilde{\mathcal{D}}_{j, \checkj}^{j}}

\newcommand{\intfToHighDualToAux}{\widetilde{\mathcal{D}}_{j}^{j, \hatj}}
\newcommand{\intfToHighDualFromAux}{\widetilde{\mathcal{D}}_{j,\hatj}^{\hatj}}
\newcommand{\intfToLowDualToAux}{\widetilde{\mathcal{D}}_{j}^{j, \checkj}}
\newcommand{\intfToLowDualFromAux}{\widetilde{\mathcal{D}}_{j,\checkj}^{\checkj}}

\newcommand{\hatS}{\hat{S}_i}
\newcommand{\checkS}{\check{S}_i}
\newcommand{\tinySum}{\textstyle{\sum}}

\newcommand{\localSubdomainLII}{L^2(\Omega_i)}
\newcommand{\localSubdomainLIIVector}{\mathbf{L}^2(\Omega_i)}
\newcommand{\localInterfaceLII}{L^2(\Gamma_j)}

\newcommand{\globalSubdomainLII}{L^2(\Omega)}
\newcommand{\globalSubdomainLIIVector}{\mathbf{L}^2(\Omega)}
\newcommand{\globalInterfaceLII}{L^2(\Gamma)}

\newcommand{\localHI}{H^1(\Omega_i)}
\newcommand{\localHIO}{H^1_0(\Omega_i)}
\newcommand{\globalHI}{H^1(\Omega)}
\newcommand{\globalHIO}{H^1_0(\Omega)}

\newcommand{\localHdiv}{\mathbf{H}(\mathrm{div},\Omega_i)}
\newcommand{\localV}{\mathbf{V}_i}
\newcommand{\localVO}{\mathbf{V}_{0,i}}
\newcommand{\localX}{\mathbf{X}_{i}}

\newcommand{\globalX}{\mathbf{X}}

\newcommand{\inner}[2]{{\left({#1},{#2}\right)}}

\newcommand{\Tau}{\mathcal{T}}
\newcommand{\subdomainGrid}[1]{\Tau_{\Omega_{#1}}}
\newcommand{\interfaceGrid}{\Tau_{\Gamma_{j}}}
\newcommand{\internalBoundaryGrid}[2]{\Tau_{\partial_{#1}\Omega_{#2}}}

\title{A posteriori error estimates for mixed-dimensional Darcy flow using non-matching grids}
\author{Jhabriel Varela\orcidlink{0000-0003-2220-2204}\footnote{Contact e-mail: \href{mailto:jhabriel@pol.una.py}{jhabriel@pol.una.py}}$^{~}$\footnote{Polytechnic School, National University of Asunción, Paraguay}$^{~}$\footnote{Polytechnic University Taiwan-Paraguay, Paraguay} \and Christian E. Schaerer\orcidlink{0000-0002-0587-7704}$^{\dagger}$ \and Eirik Keilegavlen\orcidlink{0000-0002-0333-9507}\footnote{Center for Modeling of Coupled Subsurface Dynamics, Department of Mathematics, University of Bergen, Norway}\and Inga Berre\orcidlink{0000-0002-0212-7959}$^\S$}
\date{\today}

\begin{document}

% Title page
\maketitle
\begin{abstract}
\thispagestyle{empty}
In this article, we extend the \textit{a posteriori} error estimates for hierarchical mixed-dimensional elliptic equations developed in [Varela et al., \textit{J. Numer. Math.}, 48 (2023), pp. 247–280] to the setting of non-matching mixed-dimensional grids. The extension is achieved by introducing transfer grids between the planar subdomain and interface grids, together with stable discrete projection operators for primal (potential) and dual (flux) variables. The proposed non-matching estimators remain fully guaranteed and computable. Numerical experiments, including three-dimensional problems based on community benchmarks for incompressible Darcy flow in fractured porous media, demonstrate reliable performance of the estimators for the non-matching grids and effectivity that is comparable to the estimators for matching grids. \vspace{3mm} \vspace{3mm} 
\end{abstract}
\newpage

% Highlights and TOC
\thispagestyle{empty}
% \section*{Highlights}
% \begin{itemize}
%     \item First highlight.
%     \item Second highlight.
%     \item Third highlight.
% \end{itemize}

\tableofcontents
\newpage

% Main body of the manuscript

\section{Introduction\label{sec:intro}}

The scalar hierarchical mixed-dimensional (mD) elliptic equation extends the standard Poisson equation by incorporating the coupling between domains of co-dimension one \cite{boon2020functional}. This framework is particularly well suited for modeling flow processes in fractured porous media \cite{martin2005modeling, nordbotten2019unified}, where the rock matrix is represented as a 3D host domain, fractures as embedded 2D planar surfaces, fracture intersections as 1D line segments, and their intersections as 0D points. Over the past two decades, this model has been extensively studied, and a variety of techniques have been proposed to solve it; we refer the reader to the review articles \cite{berre2019flow, formaggia2021numerical} and community benchmarks \cite{flemisch2018benchmarks, flemisch2020verification} for comprehensive overviews. Notably, the model has been recently validated experimentally \cite{Both_Brattekås_Keilegavlen_Fernø_Nordbotten_2024}, further establishing its relevance and applicability.

Given that mD problems naturally span subdomains of different dimensions, it is both geometrically and computationally advantageous to discretize each subdomain with an independent grid, avoiding the constraints of matching grids. Previous works that have considered modeling and discretization of flow in fractured porous media on non-matching grids include \cite{frih2012modeling,  dangelo2012MFEM, fumagalli2013NonMaching, fumagalli2019dual, boon2020nonmatching}. In parallel, \textit{a posteriori} error estimates were also developed for problems in mixed dimensions. For instance, \cite{pencheva2013mortar} in the context of multi-scale mortar methods, \cite{chen2016adaptive, chen2017residual, mghazli2019fractured, hecht2019residual, zhao2022adaptive} for Darcy flow in fractured media, and more recently \cite{varela2023posteriori} for hierarchical mD elliptic problems in the context of Darcy flow in networks of planar fractures. To our knowledge, however, no existing estimators cover the fully general mD setting---incorporating intersecting and fully embedded  subdomains---on non-matching grids.

Motivated by filling this gap, we extend the estimates derived in \cite{varela2023posteriori} to the non-matching case. In our context, \textit{non-matching} means that elements in the higher-dimensional grid, the lower-dimensional grid, or both, do not geometrically align with the elements of the coupling interface grid (also known as a mortar grid). However, it is important to remark that all these grids are individually strictly non-overlapping and conforming. Figure~\ref{fig:non-matching-grids} illustrates this concept in two dimensions. Non-matching grids enable independent refinement or coarsening of subdomains and interfaces, potentially leading to significant computational savings. For instance, one may refine only the fracture grid in response to large local estimator values, while leaving the surrounding matrix grid unchanged. Such fully disjoint adaptivity avoids the technical complexity of enforcing geometric inter-dimensional conformity.

\begin{figure}
    \centering
    \includegraphics[width=\linewidth]{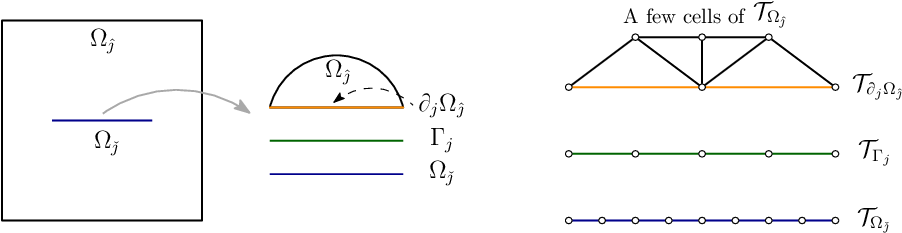}
    \caption{Mixed-dimensional geometric setting. Left: A fracture $\Omega_{\checkj}$ fully embedded in the matrix $\Omega_{\hatj}$. A close-up of the upper side of the fracture shows the interface $\Gamma_j$ coupling $\Omega_{\checkj}$ and $\Omega_{\hatj}$ through its internal boundary $\partial_j\Omega_{\hatj}$. There exists a similar coupling on the lower side of the fracture, not shown in the figures. Right: A non-matching coupling between $\internalBoundaryGrid{j}{\hatj}$, $\interfaceGrid$ and $\subdomainGrid{\checkj}$. Grids are formally defined in Section~\ref{sec:grid}.}
    \label{fig:non-matching-grids}
\end{figure}

Extending the validity of the \textit{a posteriori} error estimators from the matching to the general non-matching case is a non-trivial task, as one needs to introduce appropriate mappings of the primal (potentials) and dual (fluxes) variables between subdomains and interfaces while preserving functional conformity in a weak sense. In this work, we overcome this challenge by constructing a \emph{transfer grid} (common refinement grid between subdomains and interfaces), from which potentials and interface fluxes are projected in an energy-consistent manner.
This approach allows the estimates of \cite{varela2023posteriori} to carry over directly to the non-matching setting. Optimal generation of transfer grids, however, remains a challenging task \cite{gander2009algorithm, gander2013algorithm, mccoid2022provably}.

The remainder of the article is organized as follows. In Section~\ref{sec:md_model}, we present the mD strong formulations. Section~\ref{sec:fun_framework} introduces the functional framework, including local and mD spaces and continuous transfer operators. Section~\ref{sec:discrete_setting} covers the discrete setting, formally introducing the grids, discrete projection operators and finite-dimensional approximations. In Section~\ref{sec:error_estimation}, we first develop the abstract and then the fully computable \textit{a posteriori} error estimates. Finally, in Section~\ref{sec:numerics}, we present the numerical results and in Section~\ref{sec:conclusion} we draw our conclusions.

\section{Mixed-dimensional setting and the model problem\label{sec:md_model}}

In this section, we introduce the mD decomposition of the domain together with the strong forms of the model problem.

\subsection{Mixed-dimensional decomposition}

Following the notation introduced in \cite{varela2023posteriori}, we consider an mD domain $Y\subset\mathbb{R}^n$, $n\in \{2, 3\}$ decomposed into strictly disjoint and flat subdomains $\Omega_i$. Each $\Omega_i$ has a known dimension $d_i = d(i)$, where $0 \leq d_i \leq n$. There are $N_\Omega$ subdomains and we denote by $I=\{1, 2, \ldots, N_\Omega\}$ the index set of all subdomains. As shown in the left panel of Figure~\ref{fig:non-matching-grids}, interfaces $\Gamma_j$ establish the connection between two neighboring subdomains that are one dimension apart. Each interface has a known dimension $d_j = d(j)$, where $0 \leq d_j \leq n-1$. There are $N_\Gamma$ interfaces and the index set of interfaces is denoted by $J=\{1, 2, \ldots, N_\Gamma\}$. For a given interface $\Gamma_j$, we use the indices $\hatj \in I$ and $\checkj\in I$ to respectively identify higher- and lower-dimensional neighboring subdomains $\Omega_{\hatj}$ and $\Omega_{\checkj}$. The portion of the boundary of $\Omega_{\hatj}$ geometrically coinciding with $\Gamma_j$ is denoted by $\partial_j\Omega_{\hatj}$. We impose the following geometric compatibility conditions for the coupling triplet $(\partial_j\Omega_\hatj, \Gamma_j, \Omega_\checkj)$.

\begin{assumption}[Geometric compatibility of the coupling triplet] Let $j\in J$ and consider the coupling triplet
\begin{equation*}
    (\partial_j\Omega_{\hatj}, \Gamma_j, \Omega_\checkj),
\end{equation*}
where $\Gamma_j$ is the interface connecting $\Omega_\hatj$ and $\Omega_\checkj$.
\begin{enumerate}[label=(\roman*)]
    \vspace{1mm}
    \item $\Omega_i \subset \mathbb{R}^{n}$ is an open, bounded Lipschitz domain of intrinsic dimension $d_i\leq n$. Its boundary $\partial\Omega_i$ (when $d_i>0$) is a $(d_i-1)$-dimensional Lipschitz domain.
    \vspace{1mm}
    \item $\Gamma_j$ is a $d_j$-dimensional Lipschitz domain with $d_j=d_\checkj$, and there exist bi-Lipschitz homeomorphisms
    \begin{equation*}
        \Psi_{\hatj}^j: \partial_j\Omega_\hatj \to \Gamma_j, \qquad \Psi_{\checkj}^j:\Omega_\checkj \to \Gamma_j,
    \end{equation*}
    with uniformly bounded Lipschitz constants and inverses.
    \vspace{1mm}
    \item By means of these maps we identify $\partial_j\Omega_\hatj$, $\Gamma_j$, and $\Omega_\checkj$ as the same geometric domain up to bi-Lipschitz equivalence.
\end{enumerate}
\label{ass:isometry}
\end{assumption}

Given that several interfaces can be associated to a single subdomain, for all $i\in I$, we introduce the interface index sets $\hatS$ and $\checkS$ to establish their connections. The set $\hatS$ contains the indices of all higher-dimensional neighboring interfaces of the subdomain $\Omega_i$. Similarly, the set $\checkS$ contains the indices of all lower-dimensional neighboring interfaces of the subdomain $\Omega_i$. For example, in Figure~\ref{fig:md_full}, $\hat{S}_5=\{4, 3\}$ and $\check{S}_5=\{10\}$.

Recalling that $A\sqcup B$ denotes the disjoint union of the sets $A$ and $B$, we define the set of all subdomains by $\Omega = \sqcup_{i\in I} \Omega_i$ and the set of all interfaces by $\Gamma = \sqcup_{j\in J}\Gamma_j$. Thus, a complete disjoint partitioning of the mD domain $Y$ is given by $Y = \Omega\sqcup\Gamma$.

To finalize the geometric description, we define the boundary of $\Omega$ as $\partial\Omega = \partial_D\Omega \cup \partial_N\Omega \cup \partial_\mathcal{I}\Omega$, where $\partial_D\Omega = \cup_{i\in I}\partial_D\Omega_i$, $\partial_N\Omega = \cup_{i\in I}\partial_N\Omega_i$ and $\partial_\mathcal{I}\Omega=\cup_{i\in I} \cup_{j\in\hatS} \partial_{j}\Omega_i$ are respectively the Dirichlet, Neumann and internal portions of $\partial\Omega$, assumed to be strictly non-overlapping. Since the problem we are interested in is purely elliptic, we require $\partial_D\Omega \neq \emptyset$. A full geometric decomposition of an mD domain $Y\subset \mathbb{R}^2$ is shown schematically in Figure~\ref{fig:md_full}.

\begin{figure}
    \centering
    \includegraphics[width=0.98\textwidth]{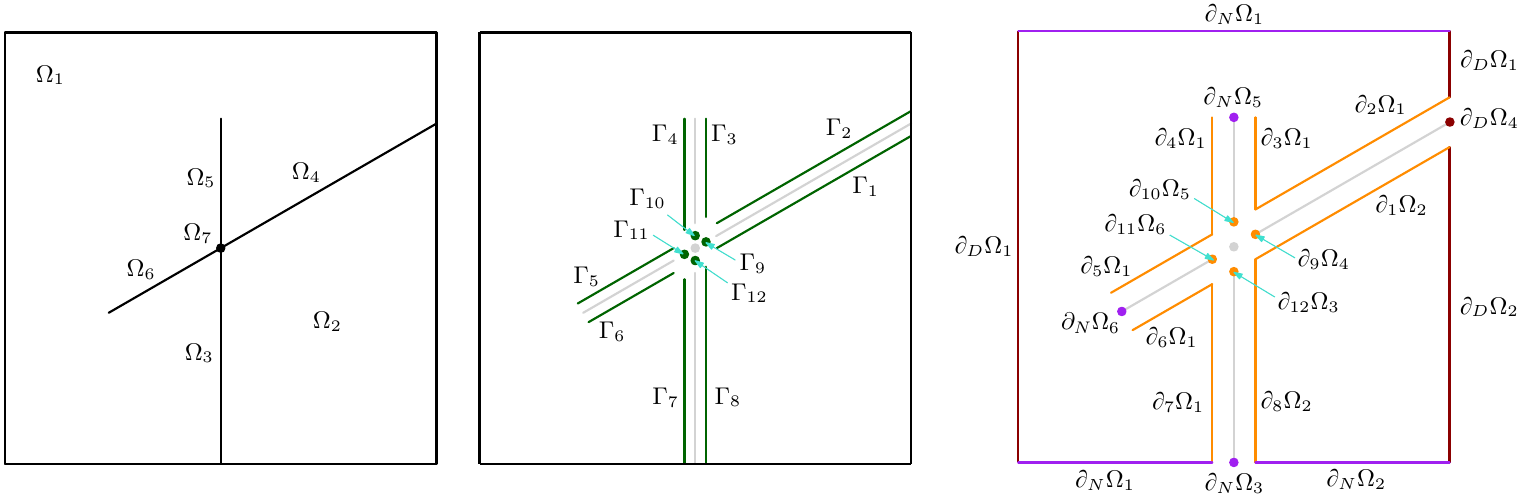}
    \caption{Complete geometric decomposition of an mD domain. Left: Subdomains. Two 2D subdomains ($\Omega_1$ and $\Omega_2$) hosting four 1D fractures ($\Omega_3$, $\Omega_4$, $\Omega_5$, $\Omega_6$) that intersect at a common 0D point ($\Omega_7$). Center: Interfaces. Twelve interfaces coupling neighboring subdomains of co-dimension one. Right: Internal and external boundary conditions.}
    \label{fig:md_full}
\end{figure}

\subsection{The model problem}

We now present the mD elliptic problem in strong form, in the context of incompressible single phase flow in fractured porous media. For the details, we refer to \cite{martin2005modeling, boon2018robust, nordbotten2019unified}.

\begin{definition}[Strong primal form] Let $i\in I$ and $j\in J$. Then, find each pressure \(p_i : \Omega_i \to \mathbb{R}\), such that
\begin{subequations}
\begin{alignat}{2}
    &-\nabla_i\cdot\left(\mathcal{K}_i\nabla_ip_i\right) + \sum_{j\in\hatS}  \kappa_j\left(p_{\checkj} - p_{\hatj}\right) = f_i && \qquad \mathrm{in}~\Omega_i, \\
    & \left(\mathcal{K}_{\hatj}\nabla_{\hatj}p_{\hatj}\right) \cdot \bm{n}_{\hatj} = \kappa_j \left(p_{\checkj} - p_{\hatj}\right) && \qquad \mathrm{on}~\partial_j\Omega_{\hatj}, \\
    &\left(\mathcal{K}_{i}\nabla_{i}p_{i}\right) \cdot \bm{n}_{i} = 0 && \qquad \mathrm{on}~ \partial_N\Omega_i, \\
    & p_i = g_{D, i} && \qquad \mathrm{on}~\partial_D\Omega_i.
\end{alignat}
\label{def:strongPrimalForm}
\end{subequations}
\end{definition}

Introducing the subdomain Darcy (tangential) flux $\bm{u}_i = -\mathcal{K}_i\nabla_ip_i$ and the interface (normal) flux $\lambda_j = -\kappa_j(p_{\checkj} - p_{\hatj})$, the strong primal form from Definition~\ref{def:strongPrimalForm} can be recast into the following dual problem.

\begin{definition}[Strong dual form]  
Let \(i \in I\) and \(j \in J\). Then, find each Darcy flux \(\bm{u}_i : \Omega_i \to \mathbb{R}^{d_i}\), interface flux \(\lambda_j : \Gamma_j \to \mathbb{R}\), and pressure \(p_i : \Omega_i \to \mathbb{R}\), such that
\begin{subequations}
    \begin{alignat}{2}
    &\nabla_i \cdot \bm{u}_i - \sum_{j \in \hatS} \lambda_j = f_i &&\qquad \mathrm{in}~\Omega_i, \label{eq:mass_conservation_dual_strong}\\
    &\bm{u}_i = -\mathcal{K}_i \nabla_i p_i &&\qquad \mathrm{in}~\Omega_i, \label{eq:tangential_darcy_law_dual_strong}\\
    &\lambda_j = -\kappa_j (p_{\checkj} - p_{\hatj}) &&\qquad \mathrm{on}~\Gamma_j, \label{eq:normal_darcy_law_dual_strong} \\
    &\bm{u}_{\hatj} \cdot \bm{n}_{\hatj} = \lambda_j &&\qquad \mathrm{on}~\partial_j \Omega_{\hatj}, \\
    &\bm{u}_i \cdot \bm{n}_i = 0 &&\qquad \mathrm{on}~\partial_N \Omega_i, \\
    &p_i = g_{D,i} &&\qquad \mathrm{on}~\partial_D \Omega_i.
\end{alignat}
\label{def:strongDualForm}
\end{subequations}
\end{definition}

In Definitions~\ref{def:strongPrimalForm} and \ref{def:strongDualForm}, $\nabla_i$ is the $d_i$-dimensional \textit{nabla} operator, $\bm{n}_i$ the outward-pointing normal vector on a given Neumann boundary $\partial_N\Omega_i$ or internal boundary $\partial_\mathcal{I}\Omega_i$, \(f_i : \Omega_i \to \mathbb{R}\) is a source term, and $g_{D,i} : \partial_D \Omega_i \to \mathbb{R}$ is the Dirichlet boundary data. Finally, $\mathcal{K}_i : \Omega_i \to \mathbb{R}^{d_i \times d_i}$ is the effective permeability matrix and \(\kappa_j : \Gamma_j \to \mathbb{R}\) is the effective normal permeability. For the physical interpretation of the model, the reader is referred to \cite{nordbotten2019unified, keilegavlen2021porepy, varela2023posteriori} and the references therein.

\begin{remark}
    The dual form provides a natural framework for coupling non-matching grids, since the introduction of the interface flux $\lambda_j$
  naturally decouples subdomain and interface contributions while strongly imposing inter-dimensional flux continuity.
\end{remark}

\section{Functional framework\label{sec:fun_framework}}

We now define the functional spaces required to formulate the weak versions of Definitions~\ref{def:strongPrimalForm} and \ref{def:strongDualForm}, respectively. From this point onward, we shall use $\nabla_i(\cdot)$ and $\nabla_i\cdot(\cdot)$ to denote the $d_i$-dimensional weak gradient and divergence operators, respectively. 

\subsection{Local and global spaces}

Let $S$ denote an open and bounded domain with boundaries satisfying Assumption~\ref{ass:isometry} when applicable. We use $L^2(S)$ to denote the space of square-integrable functions on $S$, equipped with inner product $\inner{\cdot}{\cdot}_S \equiv \inner{\cdot}{\cdot}_{L^2(S)}$ and norm $\norm{\cdot}_S \equiv \norm{\cdot}_{L^2(S)}$. We also introduce the space of square-integrable vector-valued functions on $S$, that is $\mathbf{L}^2(S) = [L^2(S)]^{d(S)}$, with $d(S)$ denoting the intrinsic dimension of $S$. For the sake of compactness we also use $ \inner{\cdot}{\cdot}_S \equiv \inner{\cdot}{\cdot}_{\mathbf{L}^2(S)}$ and $\norm{\cdot}_S \equiv \norm{\cdot}_{\mathbf{L}^2(S)}$ in the vector case.

The standard energy space is defined as $\localHI = \{q_i \in \localSubdomainLII : \nabla_i q_i \in \localSubdomainLIIVector\}$ while $\localHIO = \{q_i \in \localHI : q_i\rvert_{\partial_D\Omega_i} = 0\}$ denotes the subspace of $H^1(\Omega_i)$ of functions with vanishing traces on Dirichlet boundaries. Additionally, we let $g_i \in \localHI$ represent any function whose restriction on $\partial_D\Omega_i$ matches the boundary data, i.e., $g_i\rvert_{\partial_D\Omega_i} = g_{D,i}$. Here, $(~\cdot~)\rvert_{\partial_D\Omega_i} : \localHI\to H^{1/2}(\partial_D\Omega_i)$ is the trace operator with $H^{1/2}(\partial_D\Omega_i)$ denoting the fractional Sobolev space on the Dirichlet boundary. Although implicit, it is important to mention that the traces on the internal boundaries also belong to the space $H^{1/2}(\partial_{\mathcal{I}}\Omega_i)$.  

We now turn our focus to the flux spaces. We define the space $ \localHdiv = \{\bm{v}_i \in \localSubdomainLIIVector : \nabla_i \cdot \bm{v}_i \in \localSubdomainLII\}$ alongside the subspace 
$ \localV = \{\bm{v}_i \in \localHdiv : (\bm{v}_i \cdot \bm{n}_i)\rvert_{\partial_N\Omega_i} = 0\}$, which enforces vanishing traces on Neumann boundaries. Similarly, the space $\localVO = \{\bm{v}_i \in \localV : (\bm{v}_i \cdot \bm{n}_i)\rvert_{\partial_\mathcal{I}\Omega_i} = 0\}$
consists of functions with vanishing traces on both Neumann and internal boundaries. To handle internal boundaries, following \cite{boon2018robust}, we introduce the linear extension operator
$\mathcal{R}_j: \localInterfaceLII \to \mathbf{V}_{\hatj}$, satisfying for any $\nu_j \in \localInterfaceLII$:
\begin{equation}
(\mathcal{R}_j \nu_j)\rvert_{\partial_j\Omega_{\hatj}} \cdot \bm{n}_{\hatj} = 
\begin{cases}
    \nu_j & \text{on } \partial_j\Omega_{\hatj}, \\
    0 & \partial\Omega_\hatj \setminus \partial_j\Omega_\hatj
\end{cases}.
\end{equation}
An explicit form for $\mathcal{R}_j$ can be obtained by solving an auxiliary elliptic problem~\cite{boon2018robust}. We shall revisit this operator at the end of Section~\ref{sec:continuous_transfers}. The space $\localVO$, together with the linear extension operator $\mathcal{R}_j$, allow us to represent, in a weak sense,  subdomain fluxes $\bm{v}_i\in \localV$ with $(\bm{v}_i \cdot \bm{n}_i)\rvert_{\partial_j\Omega_{\hatj}} = \nu_j$ for all $j\in\checkS$, as $\bm{v}_i = \bm{v}_{0,i} + \tinySum_{j\in \checkS} \mathcal{R}_j\nu_j$, with $\bm{v}_{0, i} \in \localV$ and $\nu_j \in \localInterfaceLII$.
The above decomposition suggests the following local space for subdomain fluxes
\begin{equation}
    \localX := \localVO \times \prod_{j\in\checkS} \mathcal{R}_j \localInterfaceLII.
\end{equation}

\begin{remark}[$\localX$ vs. $\localV$] The space $\localX$ has higher regularity compared to $\localV$, as this latter does not contain $L^2$ traces on internal boundaries.  This enhanced regularity on internal boundaries allows the construction of pure $L^2$-type projections on coupling triplets in the case of weak dual approximations.
\end{remark}

Global spaces are now immediate by taking the Cartesian product of local spaces,
\begin{equation}
    {L^2}(\Omega) := \prod_{i\in I} L^2(\Omega_i), \qquad L^2(\Gamma) := \prod_{j\in J} L^2(\Gamma_j).
\end{equation}
Analogous definitions hold for $\mathbf{L}^2(\Omega)$, $H^1(\Omega)$, $H^1_0(\Omega)$, $H^{1/2}(\partial_D\Omega)$, $\mathbf{V}$, $\mathbf{V}_{0}$, and $\mathbf{X}$. We will later use these spaces to represent weak solutions to the model problem in the mD partition $\Omega \sqcup \Gamma$.

\subsection{Continuous transfer operators\label{sec:continuous_transfers}}

Let $j\in J$ and consider the coupling triplet $(\partial_j\Omega_{\hatj},\Gamma_j,\Omega_{\checkj})$.
Recall that by Assumption~\ref{ass:isometry}, the three sets represent the same geometric domain.

\begin{definition}[Continuous transfer operators]
We define bounded linear transfers acting as canonical identifications or embeddings:
\begin{enumerate}[label=(\roman*)]
    \item Transfers between $\partial_j\Omega_\hatj$ and $\Gamma_j$. Define 
    \begin{align*}
    \hspace{-6mm}\mathcal P_{\hatj}^j : H^{1/2}(\partial_j\Omega_{\hatj}) \to H^{1/2}(\Gamma_j), 
    ~~\mathcal D_{\hatj}^j : L^2(\partial_j\Omega_{\hatj}) \to L^2(\Gamma_j), ~~
    \mathcal D_{j}^{\hatj} : L^2(\Gamma_j) \to L^2(\partial_j\Omega_{\hatj}),
    \end{align*}
    where $\mathcal{P}_{\hatj}^j$ is an $H^{1/2}$-identification and $\mathcal{D}_\hatj^j$ and $\mathcal{D}_j^{\hatj}$ are $L^2$-identifications.
    
    \item Transfers between $\Omega_\checkj$ and $\Gamma_j$. Define
    \begin{align*}
\mathcal P_{\checkj}^j : H^1(\Omega_{\checkj}) \to H^{1/2}(\Gamma_j), ~~ 
\mathcal D_{\checkj}^j : L^2(\Omega_{\checkj}) \to L^2(\Gamma_j), ~~
\mathcal D_{j}^{\checkj} : L^2(\Gamma_j) \to L^2(\Omega_{\checkj}),
\end{align*}
where $\mathcal P_{\checkj}^j$ is the standard embedding $H^1(\Omega_{\checkj})\hookrightarrow H^{1/2}(\Gamma_j)$ on the same domain and $\mathcal{D}_\checkj^j$ and $\mathcal{D}_j^\checkj$ are $L^2$-identifications.
\end{enumerate}
\end{definition}

\begin{lemma}[Basic properties of transfer operators]
\label{lem:cont_transfers_basic}
For all admissible arguments:
\begin{enumerate}[label=(\roman*)]
    \item $L^2$ norm preservation under identification:
   \[\hspace{-6mm} \|\mathcal{D}_{\hatj}^j w \|_{\Gamma_j}\hspace{-1mm}=\|w\|_{\partial_j\Omega_{\hatj}},\quad
    \|\mathcal D_{j}^{\hatj} \nu\|_{\partial_j\Omega_{\hatj}}\hspace{-1mm}=\|\nu\|_{\Gamma_j},
    \quad
    \|\mathcal D_{\checkj}^j w\|_{\Gamma_j}\hspace{-1mm}=\|w\|_{\Omega_{\checkj}},\quad
    \|\mathcal D_{j}^{\checkj} \nu\|_{\Omega_{\checkj}}\hspace{-1mm}=\|\nu\|_{\Gamma_j}.
    \]
    \item $H^{1/2}$ norm preservation under identification:
    \[
    \|\mathcal P_{\hatj}^j q\|_{H^{1/2}(\Gamma_j)}
    =\|q\|_{H^{1/2}\left(\partial_j\Omega_{\hatj}\right)}.
    \]
    \item Sobolev embedding on the same domain:
    \[\|\mathcal P_{\checkj}^j q\|_{H^{1/2}(\Gamma_j)} \le
    C_{\mathrm{emb}}\,\|q\|_{H^1\left(\Omega_{\checkj}\right)}.
    \]
    Here, $\norm{\cdot}_{H^1(S)}$ and $\norm{\cdot}_{H^{1/2}(S)}$
are the usual Sobolev norms on the domain $S$ and $C_{\mathrm{emb}}$ is the embedding constant, cf. \cite{evans2022partial}.
\end{enumerate}
\end{lemma}

\begin{remark}[Continuous transfers vs. discrete projections]
The operators $\mathcal P$ and $\mathcal D$ above act as canonical identifications on the common geometry (with norms associated to the target space). They are introduced only to keep notation aligned with the \emph{discrete} non-matching case, where genuine projections $\widetilde{\mathcal P}$ and $\widetilde{\mathcal D}$ are required.
\end{remark}

With the interface to internal-boundary $L^2$-transfer $\mathcal D_{j}^{\hatj}(\cdot)$ defined, the subdomain extension can be written as the following composition:
\begin{equation}
\mathcal R_j = \mathcal R_j^* \circ \mathcal D_{j}^{\hat\jmath} : L^2(\Gamma_j)\to L^2(\partial_j\Omega_{\hat\jmath})\to \mathbf V_{\hat\jmath},
\end{equation}
where $\mathcal R_j^*$ is any bounded lifting from the internal boundary into $\mathbf V_{\hat\jmath}$ (cf.\ \cite{boon2018robust}).

\subsection{Weak forms\label{sec:weak_forms}}

Before stating the weak primal and dual forms of Definitions~\ref{def:strongPrimalForm} and \ref{def:strongDualForm}, we make the data regularity explicit in the following assumption. Here and in the following, we shall use the notation $[~\cdot~]$ to denote an ordered collection of subdomain and/or interface variables, depending on context.  

\begin{assumption}[Data regularity\label{ass:data}] ~~~
   \begin{enumerate}[label=(\roman*)]
       \vspace{2mm}
       \item $f = [f_i]\in L^2(\Omega)$.
       \vspace{2mm}
       \item $g_D = [g_{D, i}] \in H^{1/2}(\partial\Omega)$.
       \vspace{2mm}
       \item For all $i\in I$, $\mathcal{K}_i$ is SPD and satisfies the usual ellipticity condition
        \begin{equation*}
            c_i^2 \abs{\bm{\xi}}^2 \leq \mathcal{K}_i\bm{\xi}\cdot\bm{\xi} \leq C_i^2 \abs{\bm{\xi}}^2, \qquad\forall\bm{\xi}\in\mathbb{R}^{d_i}.
       \end{equation*}
       \item For all $j\in J$, $\kappa_j$ is non-degenerate with bounds
        \begin{equation*}
         0 < \underline{\kappa}_j \leq \kappa_j \leq \overline{\kappa} _j< \infty.   
        \end{equation*}
   \end{enumerate} 
\end{assumption}

\begin{definition}[Weak primal form] Let $p=[p_i]$ and $g = [g_i] \in \globalHI$. Then, given $f = [f_i]\in \globalSubdomainLII$, find $p \in H^1_0(\Omega) + g$ such that
\begin{align}
    &\sum_{i\in I} \left(\mathcal{K}_i\nabla_ip_i, \nabla_iq_i\right)_{\Omega_i} + \sum_{j\in J}\left(\kappa_j( \lowToIntfPrimal p_{\checkj} - \highToIntfPrimal p_{\hatj}\rvert_{\partial_j\Omega_{\hatj}}), \lowToIntfPrimal q_{\checkj} - \highToIntfPrimal q_{\hatj}\rvert_{\partial_j\Omega_{\hatj}}\right)_{\Gamma_j} \nonumber \\
    &\quad=\sum_{i\in I} \left(f_i, q_i\right)_{\Omega_i} \qquad \forall q=[q_i]\in H^1_0(\Omega).
\end{align}
\label{def:primalWeakForm}
\end{definition}

\begin{remark}[Well-posedness of the primal form]
    The bilinear form associated with the primal weak form is symmetric and, provided Assumption~\ref{ass:data} holds, coercive on $H^1_0(\Omega)$. Consequently, the problem is well-posed and equivalent to finding the unique minimizer $p \in H^1_0(\Omega) + g$ of the energy functional:
    \begin{equation*}
        \mathcal{J}(q) = \frac{1}{2}\sum_{i\in I} \norm{\mathcal{K}_i^{1/2}\nabla_i q_i}^2_{\Omega_i} 
        + \frac{1}{2}\sum_{j\in J} \norm{\kappa_j^{1/2} (\lowToIntfPrimal q_{\checkj} - \highToIntfPrimal q_{\hatj}\rvert_{\partial_j\Omega_\hatj})}^2_{\Gamma_j} 
        - \sum_{i\in I} (f_i, q_i)_{\Omega_i}.
    \end{equation*}
\end{remark}

\begin{definition}[Dual weak form]  Let $\bm{u} = [\bm{u}_{0, i}, \lambda_j]$ and $p = [p_i]$. Then, given $f=[f_i]\in \globalSubdomainLII$ and $g_D = [g_{D,i}] \in H^{1/2}(\partial_D\Omega)$, find $\bm{u} \in \mathbf{X}$ and $p\in\globalSubdomainLII$ such that

\begin{subequations}
    \begin{align}
    &\sum_{i\in I} \inner{\mathcal{K}^{-1}_i\left(\bm{u}_{0,i} + \tinySum_{j\in\checkS}\mathcal{R}^*_j\circ \intfToHighDual\lambda_j\right)}{\bm{v}_{0,i} + \tinySum_{j\in\checkS}\mathcal{R}^*_j\circ \intfToHighDual\nu_j}_{\Omega_i} \nonumber \\
    &~~- \sum_{i\in I} \inner{p_i}{\nabla_i\cdot\left(\bm{v}_{0,i}+\tinySum_{j\in\checkS}\mathcal{R}^*_j\circ \intfToHighDual\nu_j\right) + \tinySum_{j\in\hatS}\intfToLowDual\nu_j}_{\Omega_i} +\sum_{j\in J} \inner{\kappa_j^{-1}\lambda_j}{\nu_j}_{\Gamma_j} \nonumber \\
    &~~= -\sum_{i\in I} \inner{g_{D,i}}{\left(\bm{v}_{0,i} \cdot \bm{n}_i\right)\rvert_{\partial_D\Omega_{i}}}_{\partial_D\Omega_i} \qquad \forall\bm{v}=[\bm{v}_{0, i}, \nu_j] \in \mathbf{X}. \\
    &\sum_{i\in I} \inner{\nabla_i\cdot\left(\bm{u}_{0,i} + \tinySum_{j\in\checkS}\mathcal{R}^*_j\circ \intfToHighDual\lambda_j\right)}{q_i}_{\Omega_i} - \sum_{i\in I} \inner{\tinySum_{j\in\hatS}\intfToLowDual\lambda_j}{q_i}_{\Omega_i} \nonumber \\
    &~~=\sum_{i\in I} \inner{f_i}{q_i}_{\Omega_i}\qquad \forall q=[q_i]\in\globalSubdomainLII.
\end{align}
\end{subequations}
\label{def:dualWeakForm}
\end{definition}

\begin{remark}[Well-posedness of the dual form] The dual form exhibits a saddle-point structure. Well-posedness can be proven using inf-sup arguments~\cite{boon2018robust}, provided the data satisfies Assumption~\ref{ass:data}. 
\end{remark}

Later, we shall employ the energy norm to measure the \textit{a posteriori error estimates}. For an mD potential function $q=[q_i]\in\globalHIO$ we define
\begin{equation}
    \tnorm{q}^2 = \sum_{i\in I} \norm{\mathcal{K}_i^{1/2}\nabla_iq_i}_{\Omega_i}^2 + \sum_{j\in J}\norm{\kappa_j^{1/2}\left(\lowToIntfPrimal q_{\checkj} - \highToIntfPrimal q_{\hatj}\rvert_{\partial_j\Omega_{\hatj}}\right)}_{\Gamma_j}^2.
\end{equation}

For an mD flux function $\bm{v}=[\bm{v}_{0,i}, \nu_j] \in \globalSubdomainLIIVector\times\globalInterfaceLII$, we define the energy norm
\begin{equation}
    \tnormstar{\bm{v}}^2 = \sum_{i\in I} \norm{\mathcal{K}_i^{-1/2}\left(\bm{v}_{0,i}+\tinySum_{j\in\checkS}\mathcal{R}_j^*\circ \intfToHighDual\nu_j\right)}_{\Omega_i}^2 + \sum_{j\in J} \norm{\kappa_j^{-1/2}\nu_j}^2_{\Gamma_j}.
\end{equation}

We will use both norms to present the guaranteed upper bounds on the primal (the mD potential) and dual  (the mD flux) variables.

\section{Discrete setting \label{sec:discrete_setting}}

In this section, we formally introduce the grids and define the discrete projection operators. From there, we write the approximated primal and dual problems and then establish the connection between continuous and broken norms.

\subsection{Grid definitions \label{sec:grid}}

For all $i\in I$, we denote by $\subdomainGrid{i}$ a simplicial and strictly non-overlapping partition of $\Omega_i$, such that
\begin{equation*}
    \overline{\Omega_i} = \bigcup_{K\in\subdomainGrid{i}} K, \qquad \mathrm{dim}(K)\in \{0, 1, 2,3\}.
\end{equation*}
Similarly, for all $j\in J$, we denote by $\interfaceGrid$ a simplicial and strictly non-overlapping partition of $\Gamma_j$, such that
\begin{equation*}
    \overline{\Gamma_j} = \bigcup_{K\in\interfaceGrid} K, \qquad \mathrm{dim}(K)\in \{0, 1, 2\}.
\end{equation*}
Let $e$ denote a face from the set of faces $\mathcal{E}_K$ of an element $K\in \subdomainGrid{i}$. For a fixed $j\in J$, we define the internal boundary grid as
\begin{equation*}
     \internalBoundaryGrid{j}{\hatj} := 
     \left(\bigcup_{K\in\subdomainGrid{\hatj}} \bigcup_{e\in\mathcal{E}_K}e\right)
     \bigcap \partial_{j}{\Omega_{\hatj}}, \qquad \overline{\partial_{j}\Omega_{\hatj}} = \bigcup_{K\in \internalBoundaryGrid{j}{\hatj}} K, \qquad \text{dim}(K)\in\{0,1,2\}.
     \label{eq:internalBoundaryGrid}
\end{equation*}
As illustrated in the right panel of Figure~\ref{fig:non-matching-grids}, we allow the grid triplet $(\internalBoundaryGrid{j}{\hatj},\interfaceGrid, \subdomainGrid{i})$ to be non-matching. When necessary, we shall use $K_{X}$ to denote an element of any admissible grid $\mathcal{T}_X$.

To ensure the stability of the discrete projection operators that will be constructed next, we require all partitions of subdomains and interfaces to be shape-regular.

\begin{assumption}[Shape-regularity of mixed-dimensional meshes]
There exists $\vartheta>0$ such that every simplex $K$ in each mesh
$\subdomainGrid{i}$, $\interfaceGrid$, $\internalBoundaryGrid{j}{\hatj}$
satisfies $h_K/\rho_K \le \vartheta$, where $h_K$ is the diameter and $\rho_K$
the inradius of $K$. All meshes are conforming on their respective domains.
\end{assumption}

Before defining the discrete projection operators, we need to introduce the \emph{transfer grid} between subdomain and interface grids. This grid facilitates the coupling of non-matching grids by resolving the geometric mismatch between them. We focus our exposition on the 3D case (2D interface grid) as the 2D case (1D interface grid) involves fewer geometric complexities. 

\begin{definition}[Transfer grid]
Let $j\in J$ and $(\internalBoundaryGrid{j}{\hatj}, \interfaceGrid, \subdomainGrid{\checkj})$ be a coupling grid triplet. We construct the transfer grid $\Tau_{\Gamma_j, X}$ with  $X\in \{\partial_j\Omega_{\hatj}, \Omega_{\checkj}\}$ via a two-step process. First, identify the set of nodes $\mathcal{V}_{\Gamma_j, X} = \mathcal{V}_{\Gamma_j} \bigcup \mathcal{V}_{X}$ of the transfer grid, with $\mathcal{V}_{\Gamma_j}$ and $\mathcal{V}_{X}$ denoting the set of nodes of $\Tau_{\Gamma_j}$ and $\Tau_{X}$, respectively. Next, for each pair of elements $K_{\Gamma_j}$ and $K_X$, define
\begin{equation}
    S_{\Gamma_j, X} := K_{\Gamma_j} \bigcap K_{X}.
\end{equation}
The above intersections result in one of the following three cases:
\begin{enumerate}[label=(\roman*)]
    \vspace{2mm}
    \item \textit{Polygonal domain}. If  $S_{\Gamma_j, X}$ is a triangle, then $K_{\Gamma_j, X} := S_{\Gamma_j, X}$, and we  include it directly in $\Tau_{\Gamma_j, X}$.  Otherwise, $S_{\Gamma_j, X}$ is a non-simplex macro-element and it must be triangulated into simplicial elements.
    \vspace{2mm}
    \item \textit{Line segment}. If $S_{\Gamma_j, X}$ is a line-segment, it represents a shared boundary between two elements and does not contribute as an element to $\Tau_{\Gamma_j, X}$. 
    \vspace{2mm}
    \item \textit{Empty intersection}. If $S_{\Gamma_j, X}$ is void, we ignore it.
\end{enumerate}
\label{def:auxiliary_grid}
\end{definition}

Figure~\ref{fig:intersection_grids} shows examples of 1D and 2D transfer grids. To finalize this section, we summarize the geometric properties of the transfer grid.

\begin{remark}[Properties of the transfer grid] The transfer grid $\Tau_{\Gamma_j, X}$ has the properties of node inclusion, refinement, dimensionality, and partitioning:
\begin{enumerate}[label=(\roman*)]
    \vspace{2mm}
    \item $\mathcal{V}_{\Gamma_j} \subseteq \mathcal{V}_{\Gamma_j, X}$, $\mathcal{V}_X\subseteq \mathcal{V}_{\Gamma_j, X}$.
    \vspace{2mm}
    \item $\abs{\interfaceGrid}\leq \abs{\Tau_{\Gamma_j, X}}$, $\abs{\Tau_{X}}\leq \abs{\Tau_{\Gamma_j, X}}$.
    \vspace{2mm}
    \item For all $K\in\Tau_{\Gamma_j, X}, ~ \dim(K)=d_j$.
    \vspace{2mm}
    \item By Assumption~\ref{ass:isometry},  $\overline{\Gamma_j} = \bigcup_{K\in\Tau_{\Gamma_j, X}} K$,  $\overline{X} = \bigcup_{K\in\Tau_{\Gamma_j, X}} K$.  
\end{enumerate}   
\end{remark}

\begin{figure}
    \centering
    \includegraphics[width=0.85\textwidth]{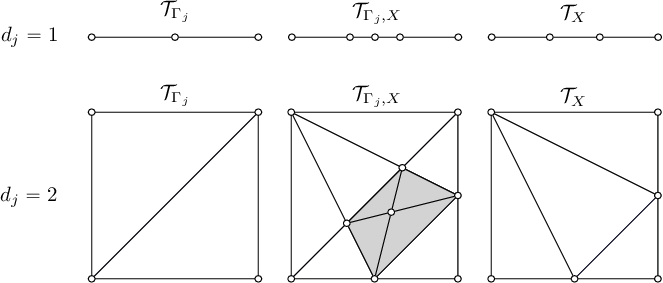}
    \caption{Transfer grids $\Tau_{\Gamma_j, X}$ coupling interface grids $\Tau_{\Gamma_j}$ with $\Tau_{X}$, $X\in\{\partial_j\Omega_{\hatj},\Omega_{\checkj}\}$. Top: 1D coupling. Bottom: 2D coupling. The shaded quadrilateral corresponds to a macro-element $S_{\Gamma_j, X}$ that was further refined into four simplices using its barycenter coordinate.}
    \label{fig:intersection_grids}
\end{figure}

\subsection{Discrete projection operators\label{sec:discrete_projection_operators}}

To define the discrete projection operators, we shall make extensive use of the broken spaces of polynomials. For $k\ge 0$ and any mesh $\Tau_X$ on a geometric set $X$, we define 
\begin{equation}
\mathbb{P}_k(\Tau_X) := \{\, q_h \in L^2(X) \;:\; q_h|_{K} \in \mathbb{P}_k(K)\ \ \forall\, K\in \Tau_X \,\}.
\label{eq:polygeneral}
\end{equation}
Functions in $\mathbb{P}_k(\Tau_X)$ are generally discontinuous across inter-element
boundaries and are not required to satisfy any external boundary condition.

\subsubsection{Discrete projections for the weak primal problem \label{sec:discrete_primal}}

The projection of potentials for the weak primal problem on non-matching grids is carried out in two stages. First, potentials defined on the source grids (the internal boundary grids and lower-dimensional subdomain grids) are prolonged onto the corresponding transfer grids. Second, the potentials are projected onto the target interface grid using the Scott-Zhang quasi-interpolant.

To be precise, let $j\in J$ be fixed and let $(\partial_{j}\Omega_{\hatj}, \Gamma_j, \Omega_{\checkj})$ be a coupling triplet. For $k\geq 1$, we define the prolongation operators 
\begin{subequations}
    \begin{align}
    \highToIntfPrimalToAux &: \mathbb{P}_k(\Tau_{\partial_j\Omega_{\hatj}})\cap H^{1/2}(\partial_j\Omega_{\hatj}) \to \mathbb{P}_k(\Tau_{\Gamma_j, \partial_j\Omega_{\hatj}}) \cap H^{1/2}(\partial_j\Omega_{\hatj}), \label{eq:primal_restriction_from_internal_boundary}\\
    \lowToIntfPrimalToAux &: \mathbb{P}_k(\Tau_{\Omega_{\checkj}}) \cap (H^1{(\Omega_{\checkj})+g_{\checkj}}) \to \mathbb{P}_k(\Tau_{\Gamma_j, \Omega_{\checkj}}) \cap (H^1{(\Omega_{\checkj})+g_{\checkj}}), \label{eq:primal_restriction_from_fracture}
\end{align}
\end{subequations}
where the transfer grids $\Tau_{\Gamma_j, \partial_j\Omega_{\hatj}}$ and $\Tau_{\Gamma_j, \Omega_{\checkj}}$ are defined as in Definition~\ref{def:auxiliary_grid}. 
Since $\Tau_{\Gamma_j, X}$ is a subdivision of $\Tau_{X}$,
for any potential $q_h \in \mathbb{P}_k(\Tau_{\partial_j\Omega_{\hatj}})\cap H^{1/2}(\partial_j\Omega_{\hatj})$, \eqref{eq:primal_restriction_from_internal_boundary} satisfies
\begin{equation}
    (\highToIntfPrimalToAux q_h)\rvert_{K_{\Gamma_j, \partial_j\Omega_{\hatj}}} = q_h \rvert_{K_{\partial_j\Omega_{\hatj}}}, \qquad \forall K_{\Gamma_j, \partial_{j}\Omega_{\hatj}} \in S_{\Gamma_{j}, \partial_j\Omega_\hatj}, \label{eq:restriction_property}
\end{equation}
for each transfer macro-element $S_{\Gamma_{j}, \partial_j\Omega_\hatj}$ and for each  $K_{\partial_{j}\Omega_{\hatj}} \subseteq S_{\Gamma_{j}, \partial_j\Omega_\hatj}$. The operator $\widetilde{\mathcal{P}}_{\checkj}^{j,\checkj}$ from \eqref{eq:primal_restriction_from_fracture} satisfies an analogous property for potentials defined on $\Tau_{\Omega_{\checkj}}$.

Once functions have been prolonged to the transfer grids, we project them onto the target interface grid:
\begin{subequations}
    \begin{align}
    \highToIntfPrimalFromAux & : \mathbb{P}_k(\Tau_{\Gamma_j, \partial_j\Omega_{\hatj}}) \cap H^{1/2}(\partial_j\Omega_{\hatj}) \to \mathbb{P}_k(\Tau_{\Gamma_j})\cap H^{1/2}(\Gamma_{j}), 
\label{eq:primal_scott_zhang_from_internal_boundary}\\
     \lowToIntfPrimalFromAux & : \mathbb{P}_k(\Tau_{\Gamma_j, \Omega_{\checkj}}) \cap (H^{1}(\Omega_{\checkj}) +g_\checkj) \to \mathbb{P}_k(\Tau_{\Gamma_j})\cap H^{1/2}(\Gamma_{j}) \label{eq:primal_scott_zhang_from_fracture}.
\end{align}
\end{subequations}
Here, $\highToIntfPrimalFromAux$ and  $\lowToIntfPrimalFromAux$ denote Scott-Zhang quasi-interpolators constructed on $\Tau_{\Gamma_j}$, see~\cite{scott1990finite} for the details.

\begin{definition}[Discrete projections for the primal problem] Let $j\in J$ and consider the coupling triplet $(\partial_{j}\Omega_{\hatj}, \Gamma_j, \Omega_{\checkj})$. For a fixed degree $k\geq 1$, the discrete projections for the weak primal problem are given by the following compositions
\begin{align}
    \highToIntfPrimalDiscrete := \highToIntfPrimalFromAux \circ \highToIntfPrimalToAux, \qquad \lowToIntfPrimalDiscrete := \lowToIntfPrimalFromAux \circ \lowToIntfPrimalToAux. \label{eq:discrete_primal_projections}
\end{align}
\end{definition}

The following lemma establishes a key stability result for the projection operators.

\begin{lemma}[Stability of discrete projection operators for the weak primal problem]
\label{lem:stability_discrete_primal}
Let $j\in J$ and $k\ge1$. Then, the operators $\widetilde{\mathcal{P}}_{\hatj}^j$ and $\widetilde{\mathcal{P}}_{\checkj}^j$ in \eqref{eq:discrete_primal_projections} are linear and bounded:
\[
\|\highToIntfPrimalDiscrete q_h\|_{H^{1/2}(\Gamma_j)} \le C_1\, \|q_h\|_{H^{1/2}(\partial_j\Omega_{\hatj})},
\qquad
\|\lowToIntfPrimalDiscrete \phi_h\|_{H^{1/2}(\Gamma_j)} \le C_2\, \|\phi_h\|_{H^{1}(\Omega_{\checkj})},
\]
with constants $C_1$ and $C_2$ depending only on the shape-regularities of the target and source grids and on the polynomial degree $k$.
\end{lemma}

\begin{proof} The prolongation operators are isometric  identifications on an equal or more refined partition of the same domain; thus norm preserving in $H^s$, $s\in [0, 1]$ (see e.g.,~\cite{lions2012non, mclean2000strongly}). The Scott-Zhang quasi-interpolant on the target interface grid is linear and $H^s$-stable for $s
\in [0, 1]$~\cite{scott1990finite}. Setting $s=1/2$ yields the first bound. For the second bound, use the Scott-Zhang stability with $s=1$ and then the continuous embedding $H^1(\Gamma_j) \hookrightarrow H^{1/2}(\Gamma_j)$, cf. Section~4.8 of \cite{brenner2008mathematical}.
\end{proof}

We emphasize that the stability results in Lemma~\ref{lem:stability_discrete_primal} (and for the ones from the next section) are independent from the transfer grids, as they act purely as data identifiers (stability constants equal to one). This result is numerically verified in Section~\ref{sec:verification}.

\subsubsection{Discrete projections of potentials for the weak dual problem\label{sec:discrete_dual_potentials}}

We now construct $L^2$–type discrete projections of potentials for the weak dual problem. These projections map traces from the internal boundary and lower-dimensional grids onto the interface grid. As in the weak primal case, we proceed in two stages: (i) elementwise prolongation to a transfer grid that resolves the geometric mismatch, and (ii) an $L^2$–projection onto the target interface grid.

Fix $j\in J$ and consider the coupling triplet $(\partial_j\Omega_{\hatj},\Gamma_j,\Omega_{\checkj})$ with transfer grids $\Tau_{\Gamma_j,\partial_j\Omega_{\hat j}}$ and $\Tau_{\Gamma_j,\Omega_{\check j}}$ from Definition~\ref{def:auxiliary_grid}. For $k\ge0$, define the prolongation operators
\begin{subequations}
    \begin{align}
  \highToIntfDualToAux &: \mathbb{P}_k(\Tau_{\partial_j\Omega_{\hat j}})
      \longrightarrow \mathbb{P}_k(\Tau_{\Gamma_j,\partial_j\Omega_{\hat j}}), \label{eq:dual_restrict_high} \\
  \lowToIntfDualToAux &: \mathbb{P}_k(\Tau_{\Omega_{\check j}})
      \longrightarrow \mathbb{P}_k(\Tau_{\Gamma_j,\Omega_{\check j}}), \label{eq:dual_restrict_low}
\end{align}
\end{subequations}
such that, for every macro-element $S_{\Gamma_j,X}$ ($X\in\{\partial_j\Omega_{\hatj},\Omega_{\checkj}\}$) and the unique source element $K_X$ with $S_{\Gamma_j,X}\subset K_X$, we have the exact restriction $(\,\cdot\,)|_{S_{\Gamma_j,X}} = (\,\cdot\,)|_{K_X}$. Thus, functions prolonged to the transfer grid are in general discontinuous across transfer elements.

Let $\{\phi_i^K\}_{i=1}^{N_k}$ be a basis of $\mathbb{P}_k(K)$ for each interface element $K\in\Tau_{\Gamma_j}$. The $L^2$–projection onto $\mathbb{P}_k(K)$ is obtained by solving the local system $M^K \alpha^K = b^K$,  where
\begin{equation}
    M_{ij}^K = (\phi_i^K,\phi_j^K)_K, \qquad 
b_i^K = (w_h,\phi_i^K)_K
    = \tinySum_{K_{\Gamma_j, X} \subset K} ~\inner{w_h}{\phi_i^K}_{K_{\Gamma_j, X}}, \label{eq:l2_proj_M_and_b}
\end{equation}
with $w_h$ denoting the prolonged function produced by
\eqref{eq:dual_restrict_high}–\eqref{eq:dual_restrict_low}
and the sum is taken over all transfer elements $K_{\Gamma_j, X} \in \Tau_{\Gamma_j, X}$ geometrically contained in $K\in \Tau_{\Gamma_j}$.

\begin{definition}[Discrete projections of potentials for the dual problem]
Let $j\in J$ and consider the coupling triplet $(\partial_{j}\Omega_{\hatj}, \Gamma_j, \Omega_{\checkj})$. For a fixed degree $k\geq 0$, the discrete projections of potentials for the weak dual problem are given by compositions
\begin{align}
  \highToIntfDualDiscrete
  ~:=~ \highToIntfDualFromAux \circ \highToIntfDualToAux, \qquad 
  \lowToIntfDualDiscrete
  ~:=~ \lowToIntfDualFromAux \circ \lowToIntfDualToAux.
\end{align}
\end{definition}

\begin{lemma}[Stability of discrete projections of potentials for the dual problem]
\label{lem:dual_potential_stability}
The operators $\highToIntfDualDiscrete$ and $\lowToIntfDualDiscrete$ are linear,
$L^2(\Gamma_j)$–stable and satisfy, for all admissible sources,
\begin{equation}
  \|\highToIntfDualDiscrete q_h\|_{\Gamma_j}
  \le C_3\, \|q_h\|_{\partial_j\Omega_{\hatj}}, \qquad  \|\lowToIntfDualDiscrete \varphi_h\|_{\Gamma_j}
  \le C_4\, \|\varphi_h\|_{\Omega_{\checkj}},
\end{equation}
where $C_3$ and $C_4$ depend only on the shape-regularities of source and target grids and on
the polynomial degree $k$.
\end{lemma}

\begin{proof}
The prolongation step acts as an isometric identification on each transfer sub-element and is therefore $L^2$-stable. The projection onto the target interface grid is the
standard elementwise $L^2$–orthogonal projector, hence locally stable and the best-approximation property holds
by orthogonality. Global bounds follow from uniform shape-regularity and the fact that, for each target element $K \in \Tau_{\Gamma_j}$, only finitely many transfer elements $K_{\Gamma_j, X} \subset K$ contribute to the local projection.
\end{proof}

\begin{remark}[Discrete projections of potentials for dual approximations of the lowest order] 
 For the lowest-order case $k = 0$, the projection amounts to elementwise averages on $\Tau_{\Gamma_j}$, i.e., a weighted sum with weights proportional to the measures of the overlaps $K_{\Gamma_j, X} \subset K$.
\end{remark}

\subsubsection{Discrete projections of interface fluxes for the weak dual problem\label{sec:discrete_dual_fluxes}}

We construct discrete projections for interface fluxes onto the internal boundary
$\Tau_{\partial_j\Omega_{\hatj}}$ and the lower-dimensional grid $\Tau_{\Omega_{\checkj}}$, while preserving the elementwise mass on the target grids. The construction consists of (i) a prolongation to a transfer grid and (ii) an agglomeration on the target grid via a constrained $L^2$ projection enforcing cell mean.

Fix $j\in J$ and the coupling triplet $(\partial_j\Omega_{\hatj}, \Gamma_j, \Omega_{\checkj})$ with transfer grids $\Tau_{\Gamma_j,\partial_j\Omega_{\hatj}}$ and $\Tau_{\Gamma_j,\Omega_{\checkj}}$ from Definition~\ref{def:auxiliary_grid}. For $k\ge 0$ and  $\nu_h\in \mathbb{P}_k(\Tau_{\Gamma_j})$, define the elementwise prolongations
\begin{subequations}
    \begin{align}
  \intfToHighDualToAux &:
  \mathbb{P}_k(\Tau_{\Gamma_j}) \to
  \mathbb{P}_k(\Tau_{\Gamma_j,\partial_j\Omega_{\hatj}}), \\
  \intfToLowDualToAux &:
  \mathbb{P}_k(\Tau_{\Gamma_j}) \to
  \mathbb{P}_k(\Tau_{\Gamma_j,\Omega_{\checkj}}),
\end{align}
\end{subequations}
by elementwise value restriction: $(\,\cdot\,)|_{S_{\Gamma_j,X}} = (\,\cdot\,)|_{K_{\Gamma_j}}$ for each macro-element $S_{\Gamma_j,X}\subset K_{\Gamma_j}$ and 
$X\in\{\partial_j\Omega_{\hatj},\Omega_{\checkj}\}$. As in Section~\ref{sec:discrete_dual_potentials}, the restricted functions are generally discontinuous across transfer elements. Given the prolonged data $\nu_h$ on $\Tau_{\Gamma_{j}, X}$ for $X \in \{\partial_j\Omega_\hatj, \Omega_\checkj\}$, we perform a constrained $L^2$ agglomeration on the target cell $K\in \Tau_{X}$. Let $\{\phi_i^K\}_{i=1}^{N_k}$ be a basis of $\mathbb{P}_k(K)$. Assemble $M_{ij}^K$ and $b_i^K$ as in \eqref{eq:l2_proj_M_and_b} and additionally define the mass constraints
\begin{equation}
    c_i^K = \inner{1}{\phi_i^K}_K, \qquad m^K= \tinySum_{K_{\Gamma_j, X} \subset K} ~\inner{\nu_h}{1}_{K_{\Gamma_j, X}}.
\end{equation}
The coefficients $\alpha^K$ and the Lagrange multiplier $\mu^K$ are determined from the local constrained optimization problem
\begin{equation}\label{eq:mc_kkt_target}
  \begin{bmatrix}
    M^K & c^K\\
    (c^K)^{\!\top} & 0
  \end{bmatrix}
  \begin{bmatrix}
    \alpha^K\\ \mu^K
  \end{bmatrix} = \begin{bmatrix}
      b^K \\ m^K
  \end{bmatrix}.
\end{equation}
The unique solvability of this system is guaranteed for any $k\ge 0$. The local mass matrix $M^K$ is SPD and since $1\in \mathbb{P}_k(K)$, the constraint vector $c^K$ is non-zero. This ensures the satisfaction of the inf-sup condition for the Lagrange multiplier $\mu^K$. Consequently, the resulting polynomial on $K$ is $(\intfToHighDualFromAux \nu_h)|_K ~=~ \sum_{i=1}^{N_k} \alpha_i^K \phi_i^K$ and analogously for $\intfToLowDualFromAux(\cdot)$ on $\Tau_{\Omega_\checkj}$.

\begin{definition}[Discrete projections of interface fluxes for the dual problem] Let $j\in J$ and consider the coupling triplet $(\partial_{j}\Omega_{\hatj}, \Gamma_j, \Omega_{\checkj})$. For a fixed degree $k\geq 0$, the discrete projections of interface fluxes for the weak dual problem are given by compositions
\begin{align}
  \intfToHighDualDiscrete
  ~:=~ \intfToHighDualFromAux \circ \intfToHighDualToAux, \qquad 
  \intfToLowDualDiscrete
  ~:=~ \intfToLowDualFromAux \circ \intfToLowDualToAux.
\end{align}
\end{definition}

The resulting mapping preserves elementwise mass and inherits stability from the constrained $L^2$ projection.

\begin{lemma}[Stability of discrete projections of interface fluxes for the dual problem]
The operators $\intfToHighDualDiscrete$ and
$\intfToLowDualDiscrete$ are linear, $L^2$–stable, and satisfy
\begin{equation}
  \|\intfToHighDualDiscrete \nu_h\|_{\partial_j\Omega_{\hatj}}
  \le C_5\, \|\nu_h\|_{\Gamma_j}, \qquad  \|\intfToLowDualDiscrete \nu_h\|_{\Omega_\checkj}
  \le C_6\, \|\nu_h\|_{\Gamma_j},
\end{equation}
with constants depending only on mesh shape-regularity and the degree $k$. Moreover, for each $K\in\Tau_{\partial_j\Omega_\hatj}$, the local polynomial $(\intfToHighDualDiscrete \nu_h)|_K$ is characterized as the unique solution of the constrained minimization problem
\begin{equation}
      (\intfToHighDualDiscrete \nu_h)|_K
  = \arg\min_{\varphi \in \mathbb{P}_k(K)}
  \bigl\{\|\varphi - \nu_h\|_{K} :
  \inner{\varphi}{1}_K = m^K \bigr\},
\end{equation}
and an analogous statement holds for $\intfToLowDualDiscrete(\cdot)$ on $\Tau_{\Omega_\checkj}$.

\end{lemma}

The argument follows verbatim from Lemma~\ref{lem:dual_potential_stability}: 
the prolongation is an $L^2$–isometry on each transfer element and the constrained 
$L^2$ projection is $L^2$–stable on shape-regular meshes. Importantly, the discrete projection operator $\intfToHighDualDiscrete$ is locally mass-conservative, in the sense that for any $\nu_h \in \mathbb{P}_k(\Tau_{\Gamma_j})$, there holds
\begin{equation}
    (\intfToHighDualDiscrete \nu_h, 1)_K = m^K
\quad\text{with}\quad
m^K = \tinySum_{K_{\Gamma_j, X} \subset K} ~ \inner{\nu_h}{1}_{K_{\Gamma_j, X}}.
\end{equation}
That is, the total interface flux over $K$ equals the sum of incoming fluxes on all transfer grid elements $K_{\Gamma_j, \partial_j\Omega_\hatj}$ contained in $K$. The same property holds for $\intfToLowDualDiscrete(\cdot)$ on $\Tau_{\Omega_\checkj}$.

\begin{remark}[Discrete projections of interface fluxes for dual approximations of the lowest-order] In the lowest-order case $k=0$, \eqref{eq:mc_kkt_target} reduces to the overlap-weighted average on
each target element:
\begin{equation*}
 (\intfToHighDualFromAux \nu_h)\rvert_K
= \frac{1}{|K|} \sum_{K_{\Gamma_j, \partial_j\Omega_\hatj} \subset K} \inner{\nu_h}{1}_{K_{\Gamma_j, \partial_j\Omega_\hatj}} = \frac{1}{|K|} \sum_{K_{\Gamma_j, \partial_j\Omega_\hatj} \subset K} |K_{\Gamma_j, \partial_j\Omega_\hatj}|~ \nu_h\rvert_{K_{\Gamma_j, \partial_j\Omega_\hatj}}.
\end{equation*}
An analogous property holds for projection onto the low-dimensional subdomain grid.
\end{remark}

In Table~\ref{tab:discrete_projections}, we summarize the six discrete projection operators introduced in this section.
\begin{table}
  \centering
  \caption{Summary of discrete projection operators of potentials and interface fluxes for the primal and dual weak problems on non-matching grids.}
  % --- local spacing tweaks ---------------------------
  \setlength{\tabcolsep}{4pt}% default is 6 pt –-→ wider cols
  \renewcommand{\arraystretch}{1.8}% default is 1 –-→ taller rows
  % -----------------------------------------------------
  \begin{tabular}{@{} c c c c c c @{}}
    \toprule
    Discrete operator & Problem & Projected quantity
    & Source grid & Target grid \\
    \midrule
    $\highToIntfPrimalDiscrete$ & Primal & Potential
    & $\internalBoundaryGrid{j}{\hatj}$ & $\interfaceGrid$ \\
    $\lowToIntfPrimalDiscrete$  & Primal & Potential
    & $\subdomainGrid{\checkj}$ & $\interfaceGrid$ \\
    $\highToIntfDualDiscrete$   & Dual  & Potential & $\internalBoundaryGrid{j}{\hatj}$ & $\interfaceGrid$ \\
    $\lowToIntfDualDiscrete$   & Dual  & Potential & $\subdomainGrid{\checkj}$ & $\interfaceGrid$ \\
    $\intfToHighDualDiscrete$    & Dual  & Flux & $\interfaceGrid$ & $\internalBoundaryGrid{j}{\hatj}$ \\
    $\intfToLowDualDiscrete$    & Dual  & Flux & $\interfaceGrid$ & $\subdomainGrid{\checkj}$ \\
    \bottomrule
  \end{tabular}
  \label{tab:discrete_projections}
\end{table}

\subsection{Primal and dual approximations\label{sec:primal_dual_approx}}

With the discrete projections operators formally defined, we can now write the approximated primal weak and dual problems. To this aim, consider first the mD broken space of polynomials on subdomains of order $k$ or less:
\begin{equation}
    \mathbb{P}_k(\Tau_{\Omega}) := \prod_{i\in I} \mathbb{P}_k(\subdomainGrid{i}).
\end{equation}

\begin{definition}[Approximated weak primal form on non-matching grids] Let $k\geq1$ be fixed and $p_h=[p_{h, i}]$, $g = [g_i] \in \globalHI$. Then, given $f = [f_i] \in L^2(\Omega)$, find $p_h \in \mathbb{P}_k(\Tau_{\Omega}) \cap ( H^1_0(\Omega) + g)$ such that
\begin{align}
    &\sum_{i\in I} \sum_{K\in\subdomainGrid{i}}\left(\mathcal{K}_i\nabla_i{p_{h,i}}, \nabla_iq_{h,i}\right)_{K}  + \sum_{j\in J}\sum_{K\in\interfaceGrid} \left(\kappa_j( \lowToIntfPrimalDiscrete p_{h,\checkj} - \highToIntfPrimalDiscrete p_{h, \hatj}), \lowToIntfPrimalDiscrete q_{h, \checkj} - \highToIntfPrimalDiscrete q_{h, \hatj}\right)_{K} \nonumber \\
    &~~~~=\sum_{i\in I} \sum_{K\in \subdomainGrid{i}}\left(f_i, q_{h,i}\right)_{K} \qquad \forall q_h=[q_{h, i}]\in \mathbb{P}_k(\mathcal{T}_\Omega)\cap H^1_0(\Omega).
\end{align}
\label{def:primalWeakFormApprox}
\end{definition}

Consider now the dual approximated problem. We wish to construct the discrete space for the mD fluxes, i.e., the discrete counterpart of $\mathbf{X}$. To this aim, we first consider the local Raviart-Thomas space given by
\begin{equation}
    \mathbb{RT}_k(\subdomainGrid{i}) = \{\bm{v}_{h, i}\in \localHdiv : \bm{v}_{h, i}\rvert_{K}\in \mathbb{RT}_k(K) ~\forall~K\in\subdomainGrid{i}\}.
\end{equation}
Then, the space $\mathbf{X}_k$ is given by
\begin{equation}
    \mathbf{X}_k = \prod_{i\in I} \left[ \left(\mathbb{RT}_k(\subdomainGrid{i}) \cap \localVO \right) \times \prod_{j\in\checkS}\widetilde{\mathcal{R}}^*_j\circ \intfToHighDualDiscrete\,\mathbb{P}_k(\interfaceGrid)\right], 
\end{equation}
where $\widetilde{\mathcal{R}}^*_j : \mathbb{P}_k(\internalBoundaryGrid{j}{\hatj})\to \mathbb{RT}_k(\subdomainGrid{\hatj})$ is the discrete extension operator from internal boundary grids onto the interior of higher-dimensional subdomain grids, cf. \cite{boon2020nonmatching} for a detailed discussion of the construction of this operator. We can now write the approximated dual weak problem as follows.

\begin{definition}[Approximated weak dual form on non-matching grids] Let $k\geq0$ and $\bm{u}_h = [\bm{u}_{h, 0, i}, \lambda_{h, j}]$, $p_h = [p_{h, i}]$. Then, given $f=[f_i]\in \globalSubdomainLII$ and $g_D = [g_{D,i}] \in H^{1/2}(\partial_D\Omega)$, find $\bm{u}_h \in \mathbf{X}_{k}$ and $p_h\in \mathbb{P}_k(\Tau_\Omega)$ such that
\begin{subequations}
    \begin{align}
    &\sum_{i\in I} \sum_{K\in\subdomainGrid{i}}\inner{\mathcal{K}^{-1}_i\left(\bm{u}_{h,0,i} + \tinySum_{j\in\checkS}\widetilde{\mathcal{R}}^*_j\circ \intfToHighDualDiscrete\lambda_{h,j}\right)}{\bm{v}_{h, 0,i} + \tinySum_{j\in\checkS}\widetilde{\mathcal{R}}^*_j\circ \intfToHighDualDiscrete\nu_{h,j}}_{K} \nonumber \\
    &~~- \sum_{i\in I} \sum_{K\in\subdomainGrid{i}}\inner{p_{h, i}}{\nabla_i\cdot\left(\bm{v}_{h, 0,i}+\tinySum_{j\in\checkS}\widetilde{\mathcal{R}}^*_j\circ \intfToHighDualDiscrete\nu_{h, j}\right) + \tinySum_{j\in\hatS}\intfToLowDualDiscrete\nu_{h,j}}_{K}  \nonumber \\
    &~~ +\sum_{j\in J} \sum_{K\in\interfaceGrid}\inner{\kappa_j^{-1}\lambda_{h,j}}{\nu_{h,j}}_{K} \nonumber \\
    &~~= -\sum_{i\in I} \sum_{K\in\Tau_{\partial_D\Omega_i}} \inner{g_{D,i}}{\left(\bm{v}_{h, 0,i} \cdot \bm{n}_i\right)\rvert_K}_{K} \qquad \forall\bm{v}_h=[\bm{v}_{h, 0, i}, \nu_{h, j}] \in \mathbf{X}_k. \\
    &\sum_{i\in I} \sum_{K\in\subdomainGrid{i}}\inner{\nabla_i\cdot\left(\bm{u}_{h, 0,i} + \tinySum_{j\in\checkS}\widetilde{\mathcal{R}}^*_j\circ \intfToHighDualDiscrete\lambda_{h, j}\right)}{q_{h, i}}_{K} \nonumber \\
    &~~ - \sum_{i\in I} \sum_{K\in\subdomainGrid{i}} \inner{\tinySum_{j\in\hatS}\intfToLowDualDiscrete\lambda_{h, j}}{q_{h, i}}_{K} \nonumber \\
    &~~=\sum_{i\in I} \sum_{K\in\subdomainGrid{i}} \inner{f_i}{q_{h, i}}_{K}\qquad \forall q_h=[q_{h, i}]\in\mathbb{P}_k(\Tau_{\Omega}).
\end{align}
\label{def:dualWeakFormApproximated}
\end{subequations}
\end{definition}
In Definition~\ref{def:dualWeakFormApproximated}, the Dirichlet boundary grid $\Tau_{\partial_D\Omega_i}$ is defined in analogy to \eqref{eq:internalBoundaryGrid}.

\subsection{Broken energy norms}

To finalize the section, we write the equivalence between the continuous and broken energy norms. For any mD function $q_h\in\mathbb{P}_k(\Tau_\Omega)\cap \globalHIO$, $k\geq1$, we define:
\begin{equation}
    \tnorm{q_h}^2 = \sum_{i\in I} \sum_{K\in\subdomainGrid{i}}\norm{\mathcal{K}_i^{1/2}\nabla_iq_{h,i}}_{K}^2 + \sum_{j\in J} \sum_{K\in\interfaceGrid} \norm{\kappa_j^{1/2}\left(\lowToIntfPrimalDiscrete q_{h, \checkj} - \highToIntfPrimalDiscrete q_{h, \hatj}\right)}_{K}^2,
\end{equation}
and for any mD function $\bm{v}_h=[\bm{v}_{h,0,i},\nu_{h,j}]\in \mathbf{X}_k$, $k\geq0$, we define:
\begin{equation}
    \tnormstar{\bm{v}_h}^2 = \sum_{i\in I} \sum_{K\in \subdomainGrid{i}}\norm{\mathcal{K}_i^{-1/2}\left(\bm{v}_{h, 0,i}+\tinySum_{j\in\checkS}\widetilde{\mathcal{R}}_{j}^*\circ \intfToHighDualDiscrete\nu_{h, j}\right)}_{K}^2 + \sum_{j\in J} \sum_{K\in\interfaceGrid} \norm{\kappa_j^{-1/2}\nu_{h, j}}^2_{K}.
\end{equation}

\section{A posteriori error estimation\label{sec:error_estimation}}

In this section, we first present the abstract \textit{a posteriori} error bounds for the case of non-matching grids. Then, after introducing the reconstruction of fluxes and potentials, we establish fully computable bounds for finite-element and mixed finite-element approximations.

\subsection{A posteriori error estimates\label{sec:error_estimates}}

Theorem 5.1 in \cite{varela2023posteriori} provides a general framework for \textit{a posteriori} error estimation in hierarchical mD elliptic problems. In this work, we focus on its application to locally mass-conservative approximations on non-matching grids, deriving explicit error bounds for the primal and dual variables. Before presenting the main theorem, we make explicit the local mass conservation property in the following assumption.

\begin{assumption}[Locally mass-conservative fluxes on non-matching grids] Let $\bm{u}_h=[\bm{u}_{h,0,i}, \lambda_{h, j}]$ be an approximated mD flux. We say $\bm{u}_h$ is locally mass-conservative if for all $K\in \Tau_{\Omega_i}$, $i\in I$, there holds
\begin{equation}
    \inner{\nabla_i\cdot (\bm{u}_{h,0,i}+\tinySum_{j\in\checkS}\widetilde{\mathcal{R}}_j^*\circ \intfToHighDualDiscrete\lambda_{h,j})-\tinySum_{j\in \hat{S}_i} \intfToLowDualDiscrete\lambda_{h,j} }{1}_K   = \inner{f_i}{1}_K.
\end{equation}
\label{ass:local_mass_conservation}
\end{assumption}

Note that to satisfy Assumption~\ref{ass:local_mass_conservation} on non-matching grids, the discrete projection operators $\intfToHighDualDiscrete$ and $\intfToLowDualDiscrete$ (cf. Section~\ref{sec:discrete_dual_fluxes}) play a key role, as their construction guarantees inter-dimensional mass conservation. Dual approximations such as the one presented in Definition~\ref{def:dualWeakFormApproximated} naturally satisfy Assumption~\ref{ass:local_mass_conservation}. However, primal approximations require more work. Usually, in a post-processing step, one has to locally equilibrate the fluxes to recover $H$(div)-conformity (see, e.g., Section~\ref{sec:flux_and_potential_rec}).

We are now ready to state the main theorem.

\begin{theorem}[Abstract \textit{a posteriori} error estimates for locally mass-conservative fluxes on non-matching grids] Define the error majorant $\mathcal{M}^{\oplus}(q, \bm{v}, f)$, valid for any $q=[q_i]\in \globalHIO + g$ and any $\bm{v}=[\bm{v}_{0, i}, \nu_{j}] \in \globalX$ satisfying Assumption~\ref{ass:local_mass_conservation}, as
\begin{equation}
    \mathcal{M}^\oplus(q, \bm{v}, f) := \eta_{\mathrm{DF}}(q, \bm{v}) + \eta_{\mathrm{R}}(\bm{v}, f),
\end{equation}
where
\begin{subequations}
\begin{align}
    \eta_{\mathrm{DF}}^2(q, \bm{v}) &= \sum_{i\in I}\sum_{K\in \subdomainGrid{i}} {\eta_{\mathrm{DF}_\parallel, K}^2} + \sum_{j\in J} \sum_{K\in \interfaceGrid} \eta_{\mathrm{DF}_\perp, K}^2, \\
    \eta_{\mathrm{R}}^2(\bm{v}, f) &= \sum_{i\in I}\sum_{K\in \subdomainGrid{i}}\eta_{\mathrm{R}, K}^2,
\end{align}
\end{subequations}
are respectively the square of the global diffusive and residual errors, with
\begin{subequations}
\begin{align}
    \eta_{\mathrm{DF}_\parallel, K} &= \norm{\mathcal{K}^{-1/2}(\bm{v}_{0, i}+\tinySum_{j\in\checkS}\widetilde{\mathcal{R}}_j^*\circ\intfToHighDualDiscrete\nu_{j})+\mathcal{K}^{1/2}\nabla_i q_i}_K, \\
    \eta_{\mathrm{DF}_\perp, K} &= \norm{\kappa^{-1/2}_j\nu_j + \kappa^{1/2}_j (\lowToIntfPrimalDiscrete q_{\checkj} - \highToIntfPrimalDiscrete q_{\hatj})}_K, \\
    \eta_{\mathrm{R}, K} &= \frac{h_K}{\pi c_{\mathcal{K}_i,K}}\norm{f_i - \nabla \cdot (\bm{v}_{0, i}+\tinySum_{j\in\checkS}\widetilde{\mathcal{R}}_j^*\circ\intfToHighDualDiscrete\nu_{j}) + \tinySum_{j\in\hatS}\intfToLowDualDiscrete \nu_j}_K,
\end{align}
\end{subequations}
respectively denoting the local subdomain diffusive, local interface diffusive, and local residual errors and with $c_{\mathcal{K}, K}$ denoting the smallest eigenvalue of the permeability matrix $\mathcal{K}_i$ restricted to the element $K$. The following a posteriori error estimates hold:

\begin{enumerate}[label=(\roman*)]
    \item Guaranteed upper bound for the primal mD variable.
    
    Let $p = [p_i] \in \globalHIO + g$ be the solution to the weak primal form from Definition~\ref{def:primalWeakForm} and $q = [q_i] \in  \mathbb{P}_k(\Tau_{\Omega}) \cap (\globalHIO + g)$ for $k\geq 1$ be arbitrary. Then, for any $\bm{v} = [\bm{v}_{0,i}, \lambda_{j}] \in \mathbf{X}_k$ with $k\geq 0$ satisfying Assumption~\ref{ass:local_mass_conservation}, there holds
    \begin{equation}
        \tnorm{p-q} \leq \mathcal{M}^\oplus(q, \bm{v}, f) =: \mathcal{M}^{\oplus}_p,
    \end{equation}
    where $\mathcal{M}^{\oplus}_p$ is the upper bound for the primal variable.

    \item Guaranteed upper bound for the dual mD variable.
    
    Let $\bm{u}=[\bm{u}_{0, i}, \lambda_j] \in \mathbf{X}$ be the solution to the weak dual form from Definition \ref{def:dualWeakForm} and $\bm{v}=[\bm{v}_{0,i}, \nu_{j}] \in \mathbf{X}_k$ satisfying Assumption~\ref{ass:local_mass_conservation} with $k\geq 0 $ be arbitrary. Then, for any $q=[q_{i}]\in \mathbb{P}_k(\mathcal{T}_\Omega) \cap \globalHIO + g$ with $k\geq 1$, there holds
    \begin{equation}
        \tnormstar{\bm{u}-\bm{v}} \leq \mathcal{M}^\oplus(q, \bm{v}, f) =: \mathcal{M}^\oplus_{\bm{u}},
    \end{equation}
    where $\mathcal{M}_{\bm{u}}^\oplus$ is the upper bound for the dual variable.
\end{enumerate}
\label{thm:main_theorem}
\end{theorem}

\begin{proof}
    Due to the functional compatibility of the discrete projection operators introduced in Section~\ref{sec:discrete_projection_operators}, the proof is identical to the one presented in Appendix C from \cite{varela2023posteriori}. The proof combines techniques from \textit{a posteriori} error estimates of the functional type~\cite{repin2007mixed, repin2008posteriori} and results from mD functional analysis~\cite{boon2020functional}. 
\end{proof}

%\begin{remark}[On the nature of the estimates] The \textit{a posteriori} error bounds from Theorem~\ref{thm:main_theorem} measure the deviation between the exact weak solutions and arbitrary energy-conforming approximations. The upper bounds have two contributions: diffusive flux errors (measuring the difference between fluxes obtained from $H^1$-potentials and locally mass-conservative $\mathbf{H}$-div fluxes) and residual errors (measuring the difference between the divergence of locally mass-conservative $\mathbf{H}$-div fluxes and the exact source terms). Importantly, thanks to the locally mass-conservative nature of the fluxes (cf., Assumption~\ref{ass:local_mass_conservation}), local residual errors are fully computable.  
%\end{remark}

\subsection{Conforming fluxes and potentials\label{sec:flux_and_potential_rec}}

In order to apply the \textit{a posteriori} bounds from Theorem~\ref{thm:main_theorem},  we need an mD potential in $\globalHIO + g$ and an mD flux in $\bm{X}$ that satisfies mass conservation at the element level. Approximations to the primal problem (Definition~\ref{def:primalWeakFormApprox}) only provide the first ingredient, whereas approximations to the dual problem (Definition~\ref{def:dualWeakFormApproximated}) only provide the second ingredient. Thus, to obtain computable bounds using the finite element method one has to reconstruct the fluxes, while the mixed-finite element requires the reconstruction of the potentials. Deriving optimal reconstruction techniques, even in the lowest-order case, is non-trivial, as one has to balance accuracy with computational cost. We therefore maintain the definition of conforming fluxes and potentials as general as possible while providing references to guide the reader.

\begin{definition}[Conforming flux] We define a conforming flux as any function $\bm{\sigma}_h$ that satisfies
\begin{equation}
    \bm{\sigma}_h \in \bm{X},
\end{equation}
and Assumption~\ref{ass:local_mass_conservation}. \label{def:flux_reconstruction}
\end{definition}

Conforming fluxes are typically obtained after equilibrating the fluxes resulting of solving the discrete weak primal problem from Definition~\ref{def:primalWeakFormApprox}. 

\begin{remark}[On flux reconstruction] Classical flux reconstruction techniques include (i) \emph{superconvergent patch recovery} \cite{zz1987estimator, zienkiewicz1992superconvergent} with $H(\text{div})$-projections,   (ii) local Neumann-patch solves \cite{ainsworth1993posteriori, ern2015polynomial, smears2020simple, cai2020robust}, and (iii) hybrid explicit‐recovery approaches \cite{cai2018hybrid}. In the mD setting, flux equilibration has been limited to enforcing mass balance across 2D/1D matrix–fracture interfaces couplings \cite{chen2017fractured, hecht2019residual, mghazli2019fractured}. To our knowledge, a  general mD flux reconstruction framework has not yet been proposed.
\end{remark}

\begin{definition}[Conforming potential] We define a conforming potential as any function $s_h$ that satisfies
\begin{equation}
    s_h \in \globalHIO + g.
\end{equation}
\label{def:pontential_reconstruction}
\end{definition}

\begin{remark}[On potential reconstruction] Several techniques exist to recover continuous potentials from a non-conforming discretization in individual subdomains. Common approaches include: (i) patch-wise weighted average of cell-centered potentials (e.g., Green-Gauss, distance-weighted averages and least squares fits) \cite{repin2009aposterior, wang2019accuracy, huang2012some, sozer2014gradient, mavriplis2003revisiting}, (ii) element-wise quadratic reconstruction via local polynomial solves~\cite{vohralik2007posteriorimfemfv, vohralik2010unified}, and (iii) element-wise linear postprocessing based on lowest-order Raviart-Tomas basis functions~\cite{varela2025linear}. Since the mD formulation does not require inter-dimensional continuity of potentials, any of these methods can be applied independently to each subdomain grid. 
\end{remark}

\subsection{Fully computable a posteriori error estimates}

Computable estimates are now readily available for primal and dual approximations. 

\begin{definition}[Computable bounds for finite element approximations] Let $p_h\in \mathbb{P}_k(\Tau_{\Omega})\cap (\globalHIO + g)$ with $k\geq 1$ be the solution to the approximated primal problem from Definition~\ref{def:primalWeakFormApprox}, and $\bm{\sigma}_h\in \mathbf{X}_k$ for $k\geq 0$ be a conforming flux (cf. Definition~\ref{def:flux_reconstruction}). Then, set $q=p_h$ and $\bm{v}=\bm{\sigma}_h$ in Theorem~\ref{thm:main_theorem} to obtain fully computable bounds.
\end{definition}

\begin{definition}[Computable bounds for mixed-finite element approximations] Let $(\bm{u}_h, p_h) \in (\mathbf{X}_k\times\mathbb{P}_k(\Tau_\Omega))$ with $k\geq 0$ be the solution to the dual approximated problem from Definition~\ref{def:dualWeakFormApproximated}, and let $s_h\in \mathbb{P}_k(\Tau_\Omega) \cap \globalHIO + g$ for $k\geq 1$ be a conforming potential (cf. Definition~\ref{def:pontential_reconstruction}). Then, set $q=s_h$ and $\bm{v}=\bm{u}_h$ in Theorem~\ref{thm:main_theorem} to obtain fully computable bounds.
\end{definition}

Thanks to the functional-based construction of the abstract \textit{a posteriori} error estimates and provided that approximated potentials and fluxes are properly reconstructed, other classical discretization techniques also fall within the scope of application of Theorem~\ref{thm:main_theorem}. Important methods used in flow modeling in fractured porous media include the  cell-centered finite volume method (CCFVM) \cite{aavatsmark2002introduction, nordbotten2021introduction}---both the two-point flux approximation (TPFA) and the multi-point flux approximation (MPFA)---, the Vertex-Centered Finite Volume Method \cite{reichenberger2006mixed}, the Mimetic Finite Difference Method \cite{antonietti2016mimetic, formaggia2018analysis}, the Discontinuous Galerkin Method \cite{antonietti2019discontinuous, antonietti2022polytopic}, and the Dual Mixed Virtual Element Method (MVEM) \cite{fumagalli2019dual}.

\subsection{Error indicators\label{sec:distinguishing}}

It is often desirable to associate error estimates to an entire subdomain or interface grid. To this aim, we introduce the following error indicators. Let $\eta_{\mathrm{DF}_{\parallel}, K}$ and $\eta_{\mathrm{R}, K}$ for $K\in \Tau_{\Omega_i}$ and  $\eta_{\mathrm{DF}_{\perp}, K}$ for $K\in \Tau_{\Gamma_j}$ be the local error estimators from Theorem~\ref{thm:main_theorem}. Then, we define the subdomain error indicator $\eta_{\Omega_i}$ and the interface error indicator $\eta_{\Gamma_j}$ as
\begin{align}
       \eta_{\Omega_i}^2 := \sum_{K\in \Tau_{\Omega_i}}  \left(\eta_{\mathrm{DF}_{\parallel}, K}^2 + \eta_{\mathrm{R}, K}^2\right) \qquad\text{and}\qquad 
       \eta_{\Gamma_j}^2 := \sum_{K\in \Tau_{\Gamma_j}} \eta_{\mathrm{DF}_{\perp}, K}^2.
\end{align}

To report results of complex fracture networks, it is beneficial to aggregate subdomain and interface errors of the same dimensionality. Let $I_{d_i} \subseteq I$ and $J_{d_j} \subseteq J$ be the reduced index sets of subdomains and interfaces of dimension $d_i$ and $d_j$, respectively. We shall then use 
\begin{align}
    \eta_{\Omega^{d_i}}^2 := \sum_{i\in I_{d_i}} \eta_{\Omega_i}^2 \qquad\text{and}\qquad \eta^2_{\Gamma^{d_j}} := \sum_{i\in J_{d_j}} \eta_{\Gamma_j}^2, 
\end{align}
to designate the error indicators of all $d_i$-dimensional subdomains and all $d_j$-dimensional interfaces, respectively.

\subsection{Effectivity indices} 

To quantify the efficiency of the \textit{a posteriori} error estimates, we shall employ the following effectivity indices
\begin{equation}
    I^{\text{eff}}_p = \frac{\mathcal{M}_p^{\oplus}(s_h, \bm{\sigma}_h,f)}{\tnorm{p-s_h}}, \qquad I^{\text{eff}}_{\bm{u}} = \frac{\mathcal{M}_{\bm{u}}^\oplus(s_h, \bm{\sigma}_h, f)}{\tnormstar{\bm{u}-\bm{\sigma}_h}}.
\end{equation}

\section{Numerical examples \label{sec:numerics}}

In this section, we present numerical verifications of the proposed non-matching error estimators against manufactured solutions and their application for test cases based on community benchmarks of flow in fractured porous media. For the sake of compactness, we only perform our analysis on mD domains with 3D hosts given that cases with 2D hosts are considerably less complex.

To obtain the numerical approximations, we employ the MPFA method~\cite{aavatsmark2002introduction, nordbotten2021introduction, keilegavlen2021porepy}, which is closely related to the lowest-order mixed-finite element approximation~\cite{vohralik2006equivalence}. Fully computable estimates are obtained after two postprocessing steps:
\begin{enumerate}
    \item[(i)]  MPFA face fluxes are directly mapped to the coefficients of the local RT0 space (guaranteeing continuity of normal traces)~\cite{repin2009aposterior}. Repeating this process in each subdomain grid yields a conforming mD flux $\bm{\sigma}_h \in \mathbf{X}_0$ (cf. Definition~\ref{def:flux_reconstruction}).
    \item[(ii)] MPFA potentials are reconstructed in each subdomain grid to obtain a conforming mD potential $s_h \in \mathbb{P}_1(\Tau_\Omega)\cap\globalHIO+g$ using the potential reconstruction technique proposed in~\cite{varela2025linear}. The technique employs the local MPFA face fluxes to first approximate a cell-centered potential gradient at the cell's barycenter which is then mapped to the cell vertices and added to the cell-center potential. This process produces $\mathbb{P}_1$-nodal potentials which are generally non-conforming. Thus, the Oswald interpolator is applied as a final step to produce a conforming global linear potential $s_h \in H^1_0(\Omega) + g$ (cf. Definition~\ref{def:pontential_reconstruction}).
\end{enumerate}

All numerical examples were implemented in \texttt{PorePy}~\cite{keilegavlen2021porepy} and are available for reproduction through the extension package \texttt{mdnme}~\cite{mdnme}.

\subsection{3D/2D verification with a manufactured solution\label{sec:verification}}

\begin{figure}
  \centering 

  \includegraphics[width=0.90\textwidth]{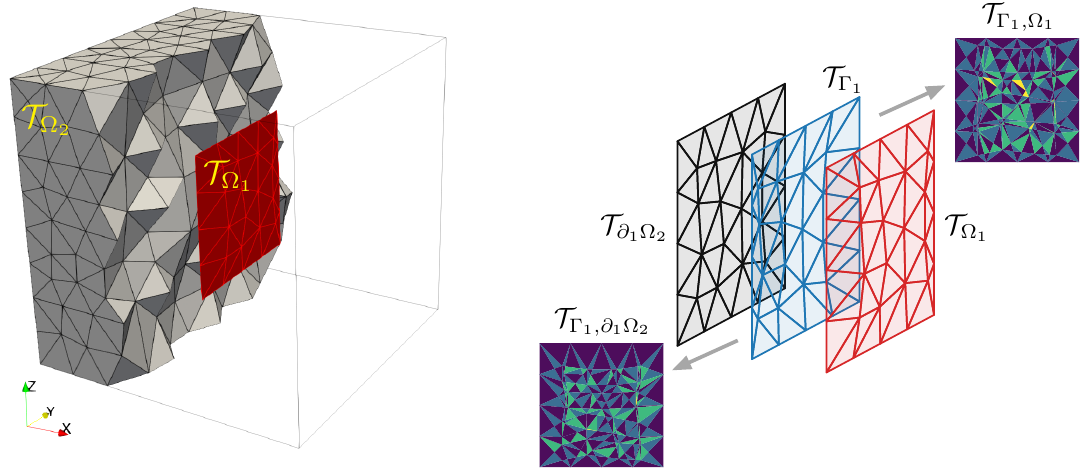}
  \caption{3D/2D numerical verification setup. Left: Mixed-dimensional grid showcasing the matrix grid $\Tau_{\Omega_2}$ and the fully embedded fracture grid $\Tau_{\Omega_1}$. Right: Exploded view of the non-matching coupling between $\Tau_{\partial_1\Omega_2}$ (black), $\mathcal{T}_{\Gamma_1}$ (blue), and $\mathcal{T}_{\Omega_1}$ (red) with the corresponding transfer grids $\Tau_{\Gamma_1, \partial_1\Omega_2}$ and $\Tau_{\Gamma_1, \Omega_1}$ (triangles are colored for visualization purposes only). Example shown for $h=0.15$ and translation $+ \bm{e}_y$.}
  \label{fig:3d_val}
\end{figure}

We employ the 3D manufactured solution from Section~7 of \cite{varela2023posteriori} to verify the efficiency of our non-matching error estimators. The mD domain consists of a unit cube $\Omega_2=(0,1)^3$ hosting a fully embedded vertical fracture
\[
\Omega_1=\{(x,y,z)\in\mathbb{R}^3:\ x=0.5,\ 0.25\le y\le 0.75,\ 0.25\le z\le 0.75\}.
\]
Interfaces $\Gamma_1$ and $\Gamma_2$ couple $\Omega_2$ to the left and right sides of $\Omega_1$, respectively (see Fig.~\ref{fig:3d_val}). With a smooth matrix pressure $p_2:\Omega_2\to\mathbb{R}$, unit tangential permeabilities $\mathcal{K}_2=\mathbf{I}_{3\times3}$ and $\mathcal{K}_1=\mathbf{I}_{2\times2}$, the exact Darcy fluxes $\bm{u}_2:\Omega_2\to\mathbb{R}^3$ and $\bm{u}_1:\Omega_1\to\mathbb{R}^2$ follow from \eqref{eq:tangential_darcy_law_dual_strong}. With unit normal permeabilities $\kappa_1=\kappa_2=1$, the interface fluxes $\lambda_1:\Gamma_1\to\mathbb{R}$ and $\lambda_2:\Gamma_2\to\mathbb{R}$ follow from \eqref{eq:normal_darcy_law_dual_strong}. Finally, \eqref{eq:mass_conservation_dual_strong} gives the exact sources $f_2:\Omega_2\to\mathbb{R}$ and $f_1:\Omega_1\to\mathbb{R}$. Exact expressions are listed in Appendix~D.2 of \cite{varela2023posteriori}.

We consider four refinement levels with target mesh sizes $h\in\{0.3,\,0.15,\,0.075,\,0.0375\}$. For each $h$, we first build a fully matching tessellation of the mD grid (reference configuration). From this, we generate non-matching grids by systematically perturbing internal nodes of the fracture grid $\mathcal{T}_{\Omega_1}$ and the interface side grids $\mathcal{T}_{\Gamma_1}$ and $\mathcal{T}_{\Gamma_2}$. Concretely, each internal node of $\mathcal{T}_{\Omega_1}$ is translated by half of the grid's mean cell diameter in the eight in-plane directions $\{\pm\bm{e}_y,\pm\bm{e}_z,\pm(\bm{e}_y+\bm{e}_z),\pm(\bm{e}_y-\bm{e}_z)\}$, while internal nodes of $\mathcal{T}_{\Gamma_1}$ and $\mathcal{T}_{\Gamma_2}$ are perturbed in the opposite directions (same magnitude). Nodes of the internal boundary grids $\mathcal{T}_{\partial_1\Omega_2}$ and $\mathcal{T}_{\partial_2\Omega_2}$ are left unperturbed. This procedure yields eight fully non-matching configurations per $h$; see the right panel of Fig.~\ref{fig:3d_val} for the left non-matching coupling corresponding to $h=0.15$ and $+\bm{e}_y$ perturbation direction.

Tables~\ref{tab:verification_majorant} and \ref{tab:verification_local} report, respectively, the majorants, true errors, and effectivity indices, and the subdomain/interface contributions, for both matching and non-matching cases. For non-matching results we show the mean over the eight perturbations together with the sample standard deviation. The results indicate that (i) the majorants and effectivity indices for the primal and dual variables remain close to the matching baseline, and (ii) the subdomain and interface contributions decrease monotonically with refinement. The tested non-matching couplings are deliberately challenging (fully non-matching with finite perturbations) to stress-test the estimators; in applications, controlled refinement typically leads to even milder mismatches.

\begin{table}
\centering
\caption{3D/2D verification results: majorants, true errors and effectivity indices.}
\label{tab:verification_majorant}
\begin{tabular}{lrrrrr}
\toprule
$h$ & $\mathcal{M}^\oplus_p=\mathcal{M}^\oplus_{\bm{u}}$ & $\tnormstar{\bm{u} - \bm{\sigma}_h}$ & $I_{\bm{u}}^{\mathrm{eff}}$ & $\tnorm{p - s_h}$ & $I_p^{\mathrm{eff}}$ \\
\midrule
\multirow{2}{*}{0.3} & 2.69e-1 & 7.10e-2 & 3.78 & 2.34e-1 & 1.15 \\
 & 2.70e-1\,$\pm$\,1.34e-3 & 7.06e-2\,$\pm$\,4.29e-4 & 3.83\,$\pm$\,2.0e-2 & 2.36e-1\,$\pm$\,6.07e-4 & 1.15\,$\pm$\,4.6e-3 \\
\midrule
\multirow{2}{*}{0.15} & 1.76e-1 & 4.56e-2 & 3.86 & 1.60e-1 & 1.10 \\
 & 1.77e-1\,$\pm$\,4.20e-4 & 4.63e-2\,$\pm$\,7.99e-4 & 3.81\,$\pm$\,5.7e-2 & 1.60\,$\pm$\,3.17e-4 & 1.10\,$\pm$\,6.0e-4 \\
\midrule
\multirow{2}{*}{0.075} & 9.28e-2 & 2.78e-2 & 3.34 & 8.48e-2 & 1.09 \\
 & 9.33e-2\,$\pm$\,3.04e-4 & 2.90e-2\,$\pm$\,1.32e-3 & 3.22\,$\pm$\,1.4e-1 & 8.54e-2\,$\pm$\,5.31e-4 & 1.09\,$\pm$\,3.3e-3 \\
\midrule
\multirow{2}{*}{0.0375} & 4.75e-2 & 1.86e-2 & 2.56 & 4.59e-2 & 1.04 \\
 & 4.77e-2\,$\pm$\,2.11e-4 & 1.64e-2\,$\pm$\,2.37e-3 & 2.96\,$\pm$\,4.2e-1 & 4.51e-2\,$\pm$\,8.78e-4 & 1.06\,$\pm$\,1.6e-2 \\
\bottomrule
\end{tabular}
\end{table}

\begin{table}
\centering
\caption{3D/2D verification results: Subdomain and interface errors.}
\label{tab:verification_local}
\begin{tabular}{lrrrr}
\hline
$h$ & $\eta_{\Omega_2}$ & $\eta_{\Omega_1}$ & $\eta_{\Gamma_1}$ & $\eta_{\Gamma_2}$ \\
\toprule
\multirow{2}{*}{0.3} & 2.43e-1 & 8.37e-3 & 6.26e-3 & 6.18e-3 \\
 & 2.43e-1\,$\pm$\,1.87e-4 & 5.47e-3\,$\pm$\,5.83e-3 & 2.28e-2\,$\pm$\,4.91e-3 & 2.24e-2\,$\pm$\,4.13e-3 \\
\midrule
\multirow{2}{*}{0.15} & 1.66e-1 & 1.62e-5 & 1.35e-3 & 1.35e-3 \\
 & 1.66e-1\,$\pm$\,8.33e-5 & 3.20e-3\,$\pm$\,3.44e-3 & 5.77e-3\,$\pm$\,1.17e-3 & 5.47e-3\,$\pm$\,1.01e-3 \\
\midrule
\multirow{2}{*}{0.075} & 8.90e-2 & 1.06e-5 & 3.89e-4 & 3.88e-4 \\
 & 8.91e-2\,$\pm$\,7.05e-5 & 1.81e-3\,$\pm$\,1.92e-3 & 4.61e-3\,$\pm$\,1.12e-3 & 4.52e-3\,$\pm$\,1.16e-3 \\
\midrule
\multirow{2}{*}{0.0375} & 4.65e-2 & 1.43e-3 & 1.13e-4 & 1.15e-4 \\
 & 4.64e-2\,$\pm$\,3.53e-5 & 8.58e-4\,$\pm$\,9.14e-4 & 2.31e-3\,$\pm$\,1.45e-3 & 2.33e-3\,$\pm$\,1.40e-3 \\
\bottomrule
\end{tabular}
\end{table}

\subsection{Regular fracture network\label{sec:geiger}}

This numerical example employs the setup of Case 2: Regular network, from the 3D flow benchmark \cite{flemisch2020verification}. The mD domain consists of a unit cube matrix hosting nine regularly oriented fractures (see the left panel of Fig.~\ref{fig:geiger}), resulting in an mD domain with one 3D subdomain, nine 2D subdomains, sixty-nine 1D subdomains, and twenty-seven 0D intersection points.

The original benchmark prescribes no-flux boundary conditions in all external faces except on the red and blue regions, where a constant Darcy flux ($\bm{u}=-1$) and a constant pressure ($p=1$) are imposed, respectively. Since our theory does not cover non-zero external boundary fluxes, and following \cite{varela2023posteriori}, we replace the inflow Neumann condition by a Dirichlet condition with $p=1$ and set the outlet pressure to $p=0$. The matrix subdomain is assigned permeability values $\mathcal{K}\in\{1, 0.1\}$, while the fractures are considered conductive ($\mathcal{K} = 1\times 10^{4}$). The complete set of parameters can be found in Table~4 from \cite{flemisch2020verification}.

The benchmark employs three levels of refinement, with approximately $500$, $1\,400$ and $32\,000$ three-dimensional cells. To generate the non-matching grids, we begin from a fully matching tessellation of the mD domain and then globally refine all fracture grids using nested bisection~\cite{geuzaine2009gmsh}. Importantly, we do not refine the 1D subdomains grids. Refining 1D subdomain grids increases the resolution around intersection points---well-known points of singularity---without improving the accuracy of interface fluxes, resulting in a larger mismatch between the interface fluxes and the (permeability-scaled) pressure jumps~\eqref{eq:normal_darcy_law_dual_strong}. We emphasize that this phenomenon is related to the limited regularity of the model itself~\cite{nordbotten2019unified}, rather to any deficiency of the error estimators.

\begin{figure}
    \centering
    \includegraphics[width=0.99\textwidth]{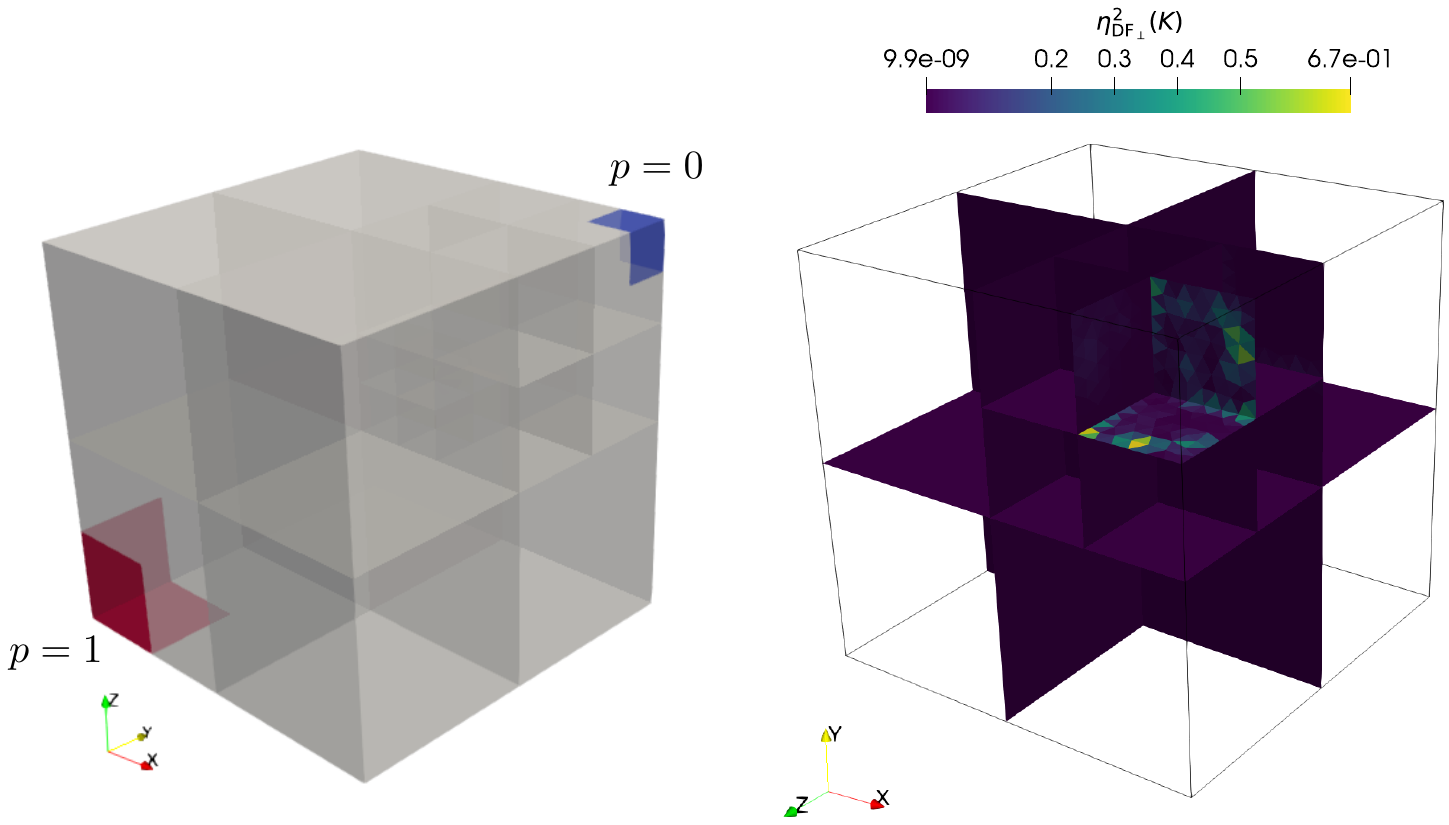}
    \caption{Regular network. Left: Geometry of the fracture network and boundary conditions. Right: Local diffusive error estimators of 2D interface grids for the finest mesh resolution.}
    \label{fig:geiger}
\end{figure}

Table~\ref{tab:geiger3dresults} reports the aggregated local errors grouped by dimensionality, together with the majorant. All quantities decrease monotonically with refinement, as expected for a stable and consistent numerical scheme. The dominant contributions to the majorant originate from the 2D interfaces, highlighting that MPFA struggles to accurately represent the interface law; this is expected considering that it is a lowest-order method. The right panel of Fig.~\ref{fig:geiger} shows the spatial distribution of the local diffusive errors. Not surprisingly, the largest errors are located around the outlet boundary region, where steep pressure gradients occur.

\begin{table}[tbp]
    \caption{Majorant and dimension-based aggregated error estimators for the 3D regular network benchmark from Section~\ref{sec:geiger}.}
    \centering
    \begin{tabular}{c c c c c c c c}
        \toprule
        Mesh & $\eta_{\Omega^3}$ & $\eta_{\Omega^2}$ & $\eta_{\Omega^1}$ & $\eta_{\Gamma^2}$ & $\eta_{\Gamma^1}$ & $\eta_{\Gamma^0}$ & $\mathcal{M}^{\oplus}$\\ 
        \midrule
        Coarse & 5.70e-1 & 3.56e-2 & 1.16e-3 & 4.35e+2 & 4.27e+0 & 8.29e-2 & 1.98e+2 \\
        \midrule
        Intermediate & 3.80e-1 & 1.92e-2 & 8.71e-4 & 6.57e+1 & 3.19e+0 & 2.32e-2 & 2.60e+1 \\
        \midrule
        Fine        & 3.24e-1 & 1.10e-2 & 6.95e-4 & 1.52e+1 & 1.41e+0 & 1.51e-2 & 6.55e+0 \\
        \bottomrule
    \end{tabular}
    \label{tab:geiger3dresults}
\end{table}

\subsection{Production on a fracture network with small features\label{sec:small_features}}

In this numerical example, we consider the setup of Case 3: Network with small features from the 3D verification benchmark  \cite{flemisch2020verification} to analyze non-matching approximations of an idealized injection/production scenario. To reduce the influence of boundary conditions on the error estimates, we modify the original specification (cf. Section~5.3 of \cite{flemisch2020verification}) and impose no-flux boundary conditions on all boundaries except the South and North faces, where we prescribe a constant Dirichlet condition $p=0$; see the left panel of Fig.~\ref{fig:thin_features}. We then add an injection well on $\Omega_2$ and a production well on $\Omega_4$, modeled as constant source terms on the well cells:
\begin{equation*}
    f_2(\mathbf{x}) = \begin{cases}
        0.1, & \mathbf{x}=(0.5, 0.65, 0.5) \\
        0, & \text{elsewhere}
    \end{cases}
    \quad\text{and}\quad
    f_4(\mathbf{x}) = \begin{cases}
       -0.1, & \mathbf{x}=(0.5, 1.6, 0.675) \\
       0, & \text{elsewhere}
    \end{cases}
    ,
\end{equation*}
where the conditions on $\mathbf{x}$ imply the cell containing said coordinate. Following the benchmark, we assign unitary permeability in the matrix $\mathcal{K}=1$ and $\mathcal{K}=1\times 10^{4}$ in the fractures, while the remaining material parameters are taken from Table~5 of \cite{flemisch2020verification}. The resulting pressure distribution is shown in the right panel of Fig.~\ref{fig:thin_features}.

\begin{figure}
    \centering
    \includegraphics[width=0.99\linewidth]{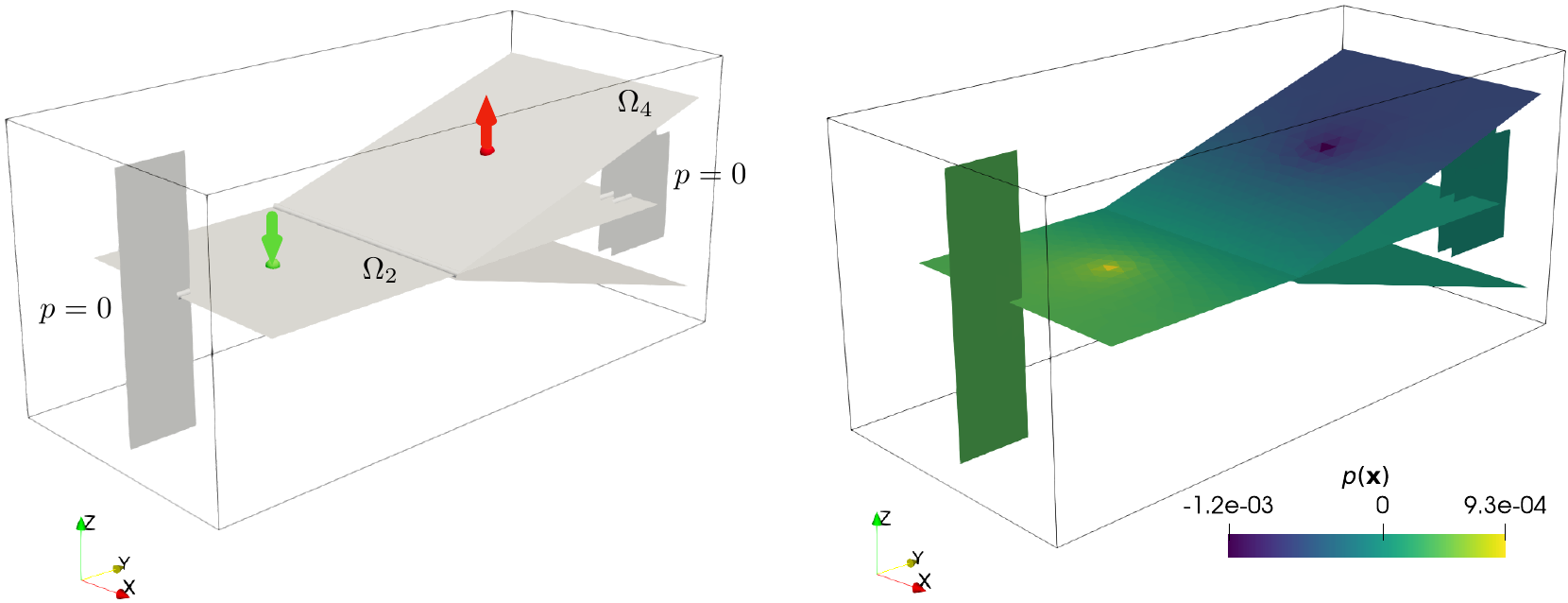}
    
    \caption{Network with small features. Left: Geometry, boundary conditions, and source/sink locations. Right: Resulting pressure distribution (shown only on 2D subdomains for visualization).}
    \label{fig:thin_features}
\end{figure}

In contrast to the previous numerical example, our goal here is not to perform a full convergence study, but to compare matching and non-matching approximations. As before, the non-matching mD grid is constructed by bisecting the matching mD grid and replacing the finer fracture grids into the coarser mD grid. 

\begin{table}[tbp]
    \caption{Majorant and dimension-based aggregated error estimators for the 3D fracture network with small features from Section~\ref{sec:small_features}.}
    \centering
    \begin{tabular}{c c c c c c c c}
        \toprule
        Approximation & $\eta_{\Omega^3}$ & $\eta_{\Omega^2}$ & $\eta_{\Omega^1}$ & $\eta_{\Gamma^2}$ & $\eta_{\Gamma^1}$ & $\mathcal{M}^{\oplus}$\\ 
        \midrule
        Matching & 8.06e-4 & 8.14e-3 & 4.96e-6 & 2.68e-1 & 2.04e-4 & 1.24e-1 \\
        \midrule
        Non-matching & 8.88e-4 & 8.11e-3 & 4.56e-6 & 2.71e-1 & 1.30e-4 & 1.26e-1 \\
        \bottomrule
    \end{tabular}
    \label{tab:thinfeatures3dresults}
\end{table}

Table~\ref{tab:thinfeatures3dresults} shows the aggregated error contributions for both discretizations. The non-matching estimators closely reproduce the matching ones, with differences at most a few percent and a relative deviation of approximately $1.2\%$ in the majorant. The non-matching approximation yields slightly larger contributions in the 3D subdomain and in the 2D interfaces, and slightly smaller contributions in the 2D and 1D subdomains as well as in the 1D interfaces. This illustrates the subtlety of non-matching refinement in mD geometries: refining fractures without fully refining interfaces may reduce errors in some components while increasing them in others. Moreover, the interface grid cannot be refined freely without simultaneously refining its adjacent internal boundary grid, as this may violate inf-sup stability conditions~\cite{boon2018robust, boon2020nonmatching}. Finally, it is important to remark that the residual contributions vanish for both approximations due to the local mass-conservative properties of MPFA~\cite{varela2023posteriori}.

\begin{figure}
    \centering
    \includegraphics[width=0.99\textwidth]{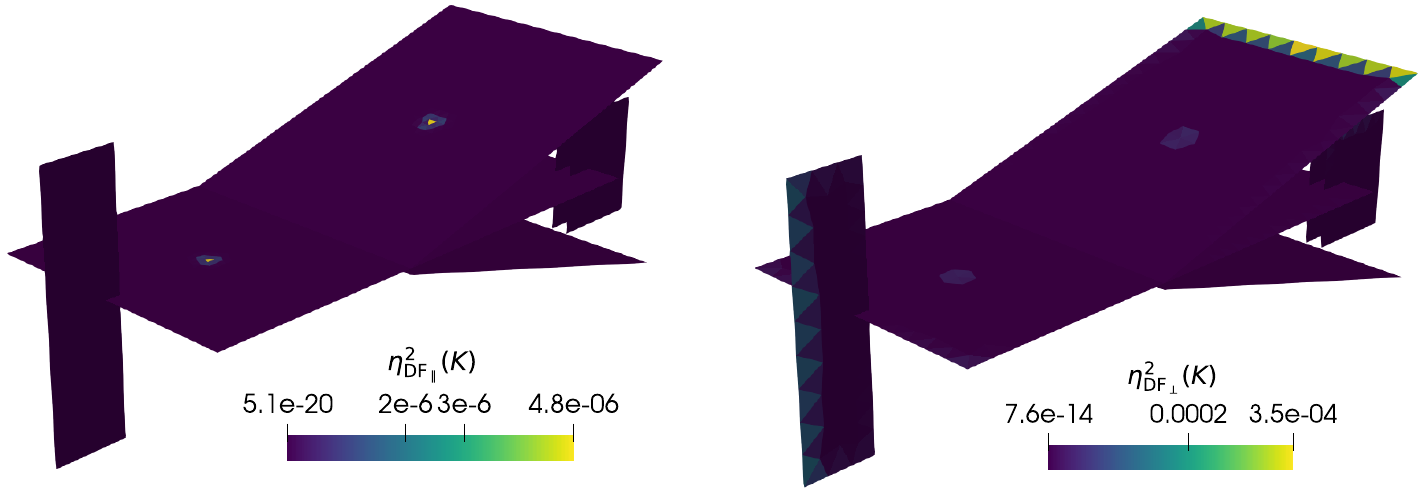}
    \caption{Non-matching local error estimators for the 3D fracture network with small features. Left: 2D diffusive errors on subdomains. Right: 2D diffusive errors on interfaces.}
    \label{fig:small_features_errors}
\end{figure}

Figure~\ref{fig:small_features_errors}  shows the spatial distribution of the 2D subdomain and interface diffusive errors for the non-matching approximation. As expected, the largest subdomain errors occur near the injection and production wells. In contrast, the interface errors are primarily concentrated along the fracture boundaries, which are known to be challenging regions for mD flow discretizations.

\section{Conclusion\label{sec:conclusion}}

In this article, we have extended the \textit{a posteriori} error estimators developed in \cite{varela2023posteriori} for matching grids to the setting of non-matching mD discretizations for elliptic problems, with particular emphasis on single-phase flow in fractured porous media. Using postprocessed MPFA approximations, we verified the validity of the proposed non-matching estimators---and, consequently, of the underlying discrete projection operators---on synthetic test cases with exact solutions as well as on 3D community benchmark problems featuring complex geometries and realistic material-parameter contrasts. The results confirm that the estimators remain reliable under non-matching refinement, while simultaneously highlighting the delicate interplay between accuracy, refinement choices, and stability requirements in mD settings.

\section*{Acknowledgments}

This project has received funding from the European Research Council (ERC) under the European Union’s Horizon 2020 research and innovation programme (grant agreement No 101002507). JV was funded by the Paraguayan National Council of Science and Technology (CONACYT) through the program:  Advanced Human Capital Insertion in Academia (PRIA01-8). CES was funded by CONACYT project number ESTR01-23 and by SISNI. The authors would like to thank Jan M. Nordbotten for valuable discussions on the topic.

% Bibliography
\bibliographystyle{unsrtnat}
\bibliography{refs.bib}

@book{evans2022partial,
  title        = {Partial Differential Equations},
  author       = {Evans, Lawrence C.},
  edition      = {2},
  series       = {Graduate Studies in Mathematics},
  volume       = {19},
  year         = {2022},
  publisher    = {American Mathematical Society},
  address      = {Providence, RI},
  isbn         = {978-1-4704-6765-9}
}

@article{Both_Brattekås_Keilegavlen_Fernø_Nordbotten_2024, place={De Bilt, The Netherlands}, title={High-fidelity experimental model verification for flow in fractured porous media}, volume={1}, url={https://ipjournal.interpore.org/index.php/interpore/article/view/31}, DOI={10.69631/ipj.v1i3nr31}, abstractNote={&amp;lt;p&amp;gt;Mixed-dimensional mathematical models for flow in fractured media have been prevalent in the modeling community for almost two decades, utilizing the explicit representation of fractures by lower-dimensional manifolds embedded in the surrounding porous media. In this work, for the first time, direct qualitative and quantitative comparisons of mixed-dimensional models are drawn against laboratory experiments. Dedicated displacement experiments of steady-state laminar flow in fractured media are investigated using both high-resolution PET images as well as state-of-the-art numerical simulations.&amp;lt;/p&amp;gt;}, number={3}, journal={InterPore Journal}, author={Both, Jakub Wiktor and Brattekås, Bergit and Keilegavlen, Eirik and Fernø, Martin and Nordbotten, Jan Martin}, year={2024}, month={Nov.}, pages={IPJ271124–6} }

@article{antonietti2022polytopic,
  title={Polytopic discontinuous {G}alerkin methods for the numerical modelling of flow in porous media with networks of intersecting fractures},
  author={Antonietti, Paola F and Facciol{\`a}, Chiara and Verani, Marco},
  journal={Computers \& Mathematics with Applications},
  volume={116},
  pages={116--139},
  year={2022},
  publisher={Elsevier},
  doi={10.1016/j.camwa.2021.08.015}
}

@Inbook{formaggia2021numerical,
author="Formaggia, Luca
and Fumagalli, Alessio
and Scotti, Anna",
editor="Gupta, Harsh K.",
title="Numerical Methods for Flow in Fractured Porous Media",
bookTitle="Encyclopedia of Solid Earth Geophysics",
year="2021",
publisher="Springer International Publishing",
address="Cham",
pages="1125--1130",
isbn="978-3-030-58631-7",
doi="10.1007/978-3-030-58631-7\_289",
url="https://doi.org/10.1007/978-3-030-58631-7\_289"
}

@article{berre2019flow,
  title={Flow in fractured porous media: A review of conceptual models and discretization approaches},
  author={Berre, Inga and Doster, Florian and Keilegavlen, Eirik},
  journal={Transport in Porous Media},
  volume={130},
  number={1},
  pages={215--236},
  year={2019},
  publisher={Springer},
  doi={10.1007/s11242-018-1171-6},
}

@article{varela2023posteriori,
  title={A posteriori error estimates for hierarchical mixed-dimensional elliptic equations},
  author={Varela, Jhabriel and Ahmed, Elyes and Keilegavlen, Eirik and Nordbotten, Jan M and Radu, Florin A},
  journal={Journal of Numerical Mathematics},
  volume={31},
  number={4},
  pages={247--280},
  year={2023},
  publisher={De Gruyter},
  doi={10.1515/jnma-2022-0038}
}

@article{keilegavlen2021porepy,
    TITLE = {Porepy: An open-source software for simulation of multiphysics processes in fractured porous media},
    AUTHOR = {Keilegavlen, E. and Berge, R. and Fumagalli, A. and Starnoni, M. and Stefansson, I. and Varela, J. and Berre, I.},
    JOURNAL = {Computational Geosciences},
    volume={25},
    number={1},
    pages={243--265},    
    YEAR = {2021},
    PUBLISHER = {Springer},
    DOI = {10.1007/s10596-020-10002-5},
}

@article {pencheva2013mortar,
    AUTHOR = {Pencheva, G. V. and Vohral\'{\i}k, M. and Wheeler, M. F. and Wildey, T.},
     TITLE = {Robust a posteriori error control and adaptivity for
              multiscale, multinumerics, and mortar coupling},
   JOURNAL = {SIAM J. Numer. Anal.},
  FJOURNAL = {SIAM Journal on Numerical Analysis},
    VOLUME = {51},
      YEAR = {2013},
    NUMBER = {1},
     PAGES = {526--554},
      ISSN = {0036-1429},
   MRCLASS = {65N30 (65N15 65N55)},
  MRNUMBER = {3033022},
MRREVIEWER = {Dietrich Braess},
       DOI = {10.1137/110839047},
}

@article{antonietti2019discontinuous,
    AUTHOR = {Antonietti, P. F. and Facciol\`a, C. and Russo, A. and Verani, M.},
     TITLE = {Discontinuous {G}alerkin approximation of flows in fractured porous media on polytopic grids},
   JOURNAL = {SIAM J. Sci. Comput.},
  FJOURNAL = {SIAM Journal on Scientific Computing},
    VOLUME = {41},
      YEAR = {2019},
    NUMBER = {1},
     PAGES = {A109--A138},
      ISSN = {1064-8275},
   MRCLASS = {65N30 (65N12 65N15 76S05)},
  MRNUMBER = {3895121},
MRREVIEWER = {JaEun Ku},
       DOI = {10.1137/17M1138194},
}

@article{martin2005modeling,
    AUTHOR = {Martin, V. and Jaffr\'{e}, J. and Roberts, J. E.},
     TITLE = {Modeling fractures and barriers as interfaces for flow in porous media},
   JOURNAL = {SIAM J. Sci. Comput.},
  FJOURNAL = {SIAM Journal on Scientific Computing},
    VOLUME = {26},
      YEAR = {2005},
    NUMBER = {5},
     PAGES = {1667--1691},
      ISSN = {1064-8275},
   MRCLASS = {76S05 (65N30 76M10)},
  MRNUMBER = {2142590},
MRREVIEWER = {Ma\l gorzata Peszy\'{n}ska},
       DOI = {10.1137/S1064827503429363},
}

@article {dangelo2012MFEM,
    AUTHOR = {D'Angelo, C. and Scotti, A.},
     TITLE = {A mixed finite element method for Darcy flow in fractured porous media with non-matching grids},
   JOURNAL = {ESAIM: Mathematical Modelling and Numerical Analysis - Mod\'elisation Math\'ematique et Analyse Num\'erique},
 PUBLISHER = {EDP-Sciences},
    VOLUME = {46},
    NUMBER = {2},
      YEAR = {2012},
     PAGES = {465-489},
       DOI = {10.1051/m2an/2011148},
       ZBK = {1271.76322},
  LANGUAGE = {en},
}

@article{hecht2019residual,
    AUTHOR = {Hecht, F. and Mghazli, Z. and Naji, I. and Roberts, J. E.},
     TITLE = {A residual {\it a posteriori} error estimators for a model for flow in porous media with fractures},
   JOURNAL = {J. Sci. Comput.},
  FJOURNAL = {Journal of Scientific Computing},
    VOLUME = {79},
      YEAR = {2019},
    NUMBER = {2},
     PAGES = {935--968},
      ISSN = {0885-7474},
   MRCLASS = {65N30 (65N15 76S05)},
  MRNUMBER = {3968997},
       DOI = {10.1007/s10915-018-0875-7},
}

@article {aavatsmark2002introduction,
    AUTHOR = {Aavatsmark, I.},
     TITLE = {An introduction to multipoint flux approximations for quadrilateral grids},
      NOTE = {Locally conservative numerical methods for flow in porous media},
   JOURNAL = {Comput. Geosci.},
  FJOURNAL = {Computational Geosciences},
    VOLUME = {6},
      YEAR = {2002},
    NUMBER = {3-4},
     PAGES = {405--432},
      ISSN = {1420-0597},
   MRCLASS = {76M25 (65N06 76S05)},
  MRNUMBER = {1956024},
       DOI = {10.1023/A:1021291114475},
}

@article{boon2018robust,
    AUTHOR = {Boon, W. M. and Nordbotten, J. M. and Yotov, I.},
     TITLE = {Robust discretization of flow in fractured porous media},
   JOURNAL = {SIAM J. Numer. Anal.},
  FJOURNAL = {SIAM Journal on Numerical Analysis},
    VOLUME = {56},
      YEAR = {2018},
    NUMBER = {4},
     PAGES = {2203--2233},
      ISSN = {0036-1429},
   MRCLASS = {65N30 (65N12 76S05)},
  MRNUMBER = {3829517},
MRREVIEWER = {Harbir Antil},
       DOI = {10.1137/17M1139102},
}

@article{nordbotten2019unified,
    AUTHOR = {Nordbotten, J. M. and Boon, W. M. and Fumagalli, A. and Keilegavlen, E.},
     TITLE = {Unified approach to discretization of flow in fractured porous media},
   JOURNAL = {Comput. Geosci.},
  FJOURNAL = {Computational Geosciences},
    VOLUME = {23},
      YEAR = {2019},
    NUMBER = {2},
     PAGES = {225--237},
      ISSN = {1420-0597},
   MRCLASS = {76S05},
  MRNUMBER = {3941939},
       DOI = {10.1007/s10596-018-9778-9},
}

@incollection {vohralik2007posteriorimfemfv,
    AUTHOR = {Vohral\'{\i}k, M.},
     TITLE = {A posteriori error estimates for finite volume and mixed finite element discretizations of convection-diffusion-reaction equations},
 BOOKTITLE = {Paris-{S}ud {W}orking {G}roup on {M}odelling and {S}cientific {C}omputing 2006--2007},
    SERIES = {ESAIM Proc.},
    VOLUME = {18},
     PAGES = {57--69},
 PUBLISHER = {EDP Sci.},
   ADDRESS = {Les Ulis},
      YEAR = {2007},
   MRCLASS = {65N30 (65N15 76M12)},
  MRNUMBER = {2404896},
MRREVIEWER = {Mario Ohlberger},
       DOI = {10.1051/proc:071806},
}

@article {vohralik2010unified,
    AUTHOR = {Vohral\'{\i}k, M.},
     TITLE = {Unified primal formulation-based a priori and a posteriori error analysis of mixed finite element methods},
   JOURNAL = {Math. Comp.},
  FJOURNAL = {Mathematics of Computation},
    VOLUME = {79},
      YEAR = {2010},
    NUMBER = {272},
     PAGES = {2001--2032},
      ISSN = {0025-5718},
   MRCLASS = {65N30 (65N15)},
  MRNUMBER = {2684353},
MRREVIEWER = {Nasser H. Sweilam},
       DOI = {10.1090/S0025-5718-2010-02375-0},
}

@article {ern2015polynomial,
    AUTHOR = {Ern, A. and Vohral\'{\i}k, M.},
     TITLE = {Polynomial-degree-robust a posteriori estimates in a unified setting for conforming, nonconforming, discontinuous {G}alerkin, and mixed discretizations},
   JOURNAL = {SIAM J. Numer. Anal.},
  FJOURNAL = {SIAM Journal on Numerical Analysis},
    VOLUME = {53},
      YEAR = {2015},
    NUMBER = {2},
     PAGES = {1058--1081},
      ISSN = {0036-1429},
   MRCLASS = {65N30 (65N15)},
  MRNUMBER = {3335498},
MRREVIEWER = {Dominic Breit},
       DOI = {10.1137/130950100},
}

@article {flemisch2018benchmarks,
     TITLE = {Benchmarks for single-phase flow in fractured porous media},
    AUTHOR = {Flemisch, B. and Berre, I. and Boon, W. and Fumagalli, A. and Schwenck, N. and Scotti, A. and Stefansson, I. and Tatomir, A.},
   JOURNAL = {Advances in Water Resources},
    VOLUME = {111},
     PAGES = {239--258},
      YEAR = {2018},
 PUBLISHER = {Elsevier},
       DOI = {10.1016/j.advwatres.2017.10.036},
}

@article {boon2020functional,
    AUTHOR = {Boon, W. M. and Nordbotten, J. M and Vatne, J. E.},
     TITLE = {Functional analysis and exterior calculus on mixed-dimensional geometries},
   JOURNAL = {Annali di Matematica Pura ed Applicata (1923-)},
     PAGES = {757–789},
      YEAR = {2021},
 PUBLISHER = {Springer},
       DOI = {10.1007/s10231-020-01013-1},
}

@article {chen2017fractured,
    AUTHOR = {Chen, H. and Sun, S.},
     TITLE = {A residual-based a posteriori error estimator for single-phase {D}arcy flow in fractured porous media},
   JOURNAL = {Numer. Math.},
  FJOURNAL = {Numerische Mathematik},
    VOLUME = {136},
      YEAR = {2017},
    NUMBER = {3},
     PAGES = {805--839},
      ISSN = {0029-599X},
   MRCLASS = {65N30 (65N12 65N15 76S05)},
  MRNUMBER = {3660303},
MRREVIEWER = {JiChun Li},
       DOI = {10.1007/s00211-016-0851-9},
}

@article {mghazli2019fractured,
    AUTHOR = {Mghazli, Z. and Naji, I.},
     TITLE = {Guaranteed {\it a posteriori} error estimates for a fractured porous medium},
   JOURNAL = {Math. Comput. Simulation},
  FJOURNAL = {Mathematics and Computers in Simulation},
    VOLUME = {164},
      YEAR = {2019},
     PAGES = {163--179},
      ISSN = {0378-4754},
   MRCLASS = {76S05 (65N15 65N30)},
  MRNUMBER = {3980194},
       DOI = {10.1016/j.matcom.2019.02.002},
}

@article {flemisch2020verification,
   TITLE = "Verification benchmarks for single-phase flow in three-dimensional fractured porous media",
 JOURNAL = "Advances in Water Resources",
  VOLUME = "147",
   PAGES = "103759",
    YEAR = "2021",
    ISSN = "0309-1708",
     DOI = "10.1016/j.advwatres.2020.103759",
  AUTHOR = "I. Berre and W. M. Boon and B. Flemisch and A. Fumagalli and D. Gläser and E. Keilegavlen and A. Scotti and I. Stefansson and A. Tatomir and K. Brenner and S. Burbulla and P. Devloo and O. Duran and M. Favino and J. Hennicker and I-H. Lee and K. Lipnikov and R. Masson and K. Mosthaf and M. G. C. Nestola and C-F. Ni and K. Nikitin and P. Schädle and D. Svyatskiy and R. Yanbarisov and P. Zulian",
}

@article {fumagalli2019dual,
      TITLE = {Dual virtual element methods for discrete fracture matrix models},
     AUTHOR = {Fumagalli, A. and Keilegavlen, E.},
    JOURNAL = {Oil \& Gas Science and Technology--Revue d’IFP Energies nouvelles},
     VOLUME = {74},
      PAGES = {41},
       YEAR = {2019},
  PUBLISHER = {EDP Sciences},
        DOI = {10.2516/ogst/2019008},
}

@article {repin2007mixed,
  AUTHOR = {Repin, S. I. and Sauter, S. and Smolianski, A.},
     TITLE = {Two-sided a posteriori error estimates for mixed formulations of elliptic problems},
   JOURNAL = {SIAM J. Numer. Anal.},
  FJOURNAL = {SIAM Journal on Numerical Analysis},
    VOLUME = {45},
      YEAR = {2007},
    NUMBER = {3},
     PAGES = {928--945},
      ISSN = {0036-1429},
   MRCLASS = {65N15 (35J25)},
  MRNUMBER = {2318795},
       DOI = {10.1137/050641533},
}

@book{repin2008posteriori,
  AUTHOR = {Repin, S. I.},
     TITLE = {A posteriori estimates for partial differential equations},
    SERIES = {Radon Series on Computational and Applied Mathematics},
    VOLUME = {4},
 PUBLISHER = {Walter de {G}ruyter {G}mb{H} \& {C}o. {KG}},
   ADDRESS = {Berlin}, 
      YEAR = {2008},
     PAGES = {xii+316},
      ISBN = {978-3-11-019153-0},
   MRCLASS = {35-02 (35A35)},
  MRNUMBER = {2458008},
MRREVIEWER = {Wen Bin Liu},
       DOI = {10.1515/9783110203042},
}

@article {repin2009aposterior,
     AUTHOR = {Cochez-Dhondt, S. and Nicaise, S. and Repin, S. I.},
      TITLE = {A posteriori error estimates for finite volume approximations},
    JOURNAL = {Math. Model. Nat. Phenom.},
   FJOURNAL = {Mathematical Modelling of Natural Phenomena},
     VOLUME = {4},
       YEAR = {2009},
     NUMBER = {1},
      PAGES = {106--122},
       ISSN = {0973-5348},
    MRCLASS = {65N15 (65N06)},
   MRNUMBER = {2483555},
 MRREVIEWER = {Venkataraman Vanaja},
        DOI = {10.1051/mmnp/20094105},
}

@article {zz1987estimator,
    AUTHOR = {Zienkiewicz, O. C. and Zhu, J. Z.},
     TITLE = {A simple error estimator and adaptive procedure for practical engineering analysis},
   JOURNAL = {Internat. J. Numer. Methods Engrg.},
  FJOURNAL = {International Journal for Numerical Methods in Engineering},
    VOLUME = {24},
      YEAR = {1987},
    NUMBER = {2},
     PAGES = {337--357},
      ISSN = {0029-5981},
   MRCLASS = {73K25},
  MRNUMBER = {875306},
       DOI = {10.1002/nme.1620240206},
}

@article{frih2012modeling,
  title = {Modeling fractures as interfaces with nonmatching grids},
  author = {Frih, N. and Martin, V. and Roberts, J. E. and Sa{\^a}da, A.},
  journal = {{C}omputational {G}eosciences},
  volume = {16},
  number = {4},
  pages = {1043--1060},
  year = {2012},
  publisher = {Springer},
  DOI = {10.1007/s10596-012-9302-6}
}

@incollection{nordbotten2021introduction,
author="Nordbotten, Jan Martin
and Keilegavlen, Eirik",
editor="Di Pietro, Daniele Antonio
and Formaggia, Luca
and Masson, Roland",
title="An Introduction to Multi-point Flux ({MPFA}) and Stress ({MPSA}) Finite Volume Methods for Thermo-poroelasticity",
bookTitle="Polyhedral Methods in Geosciences",
year="2021",
publisher="Springer International Publishing",
address="Cham",
pages="119--158",
isbn="978-3-030-69363-3",
doi="10.1007/978-3-030-69363-3\_4",
}

@article{boon2020nonmatching,
	author = {Boon, Wietse M. and Gl\"{a}ser, Dennis and Helmig, Rainer and Yotov, Ivan},
	doi = {10.1137/20M1361407},
	journal = {SIAM Journal on Numerical Analysis},
	number = {3},
	pages = {1193-1225},
	title = {Flux-Mortar Mixed Finite Element Methods on NonMatching Grids},
	url = {https://doi.org/10.1137/20M1361407},
	volume = {60},
	year = {2022},
}

@article{reichenberger2006mixed,
  title={A mixed-dimensional finite volume method for two-phase flow in fractured porous media},
  author={Reichenberger, Volker and Jakobs, Hartmut and Bastian, Peter and Helmig, Rainer},
  journal={Advances in water resources},
  volume={29},
  number={7},
  pages={1020--1036},
  year={2006},
  publisher={Elsevier},
  doi={10.1016/j.advwatres.2005.09.001}
}

@article{antonietti2016mimetic,
  title={Mimetic finite difference approximation of flows in fractured porous media},
  author={Antonietti, Paola F and Formaggia, Luca and Scotti, Anna and Verani, Marco and Verzott, Nicola},
  journal={ESAIM: Mathematical Modelling and Numerical Analysis},
  volume={50},
  number={3},
  pages={809--832},
  year={2016},
  publisher={EDP Sciences},
  doi={10.1051/m2an/2015087}
}

@article{formaggia2018analysis,
  title={Analysis of a mimetic finite difference approximation of flows in fractured porous media},
  author={Formaggia, Luca and Scotti, Anna and Sottocasa, Federica},
  journal={ESAIM: Mathematical Modelling and Numerical Analysis},
  volume={52},
  number={2},
  pages={595--630},
  year={2018},
  publisher={EDP Sciences},
  doi={10.1051/m2an/2017028}
}

@article{vohralik2006equivalence,
  title={Equivalence between lowest-order mixed finite element and multi-point finite volume methods on simplicial meshes},
  author={Vohral{\'\i}k, Martin},
  journal={ESAIM: Mathematical Modelling and Numerical Analysis},
  volume={40},
  number={2},
  pages={367--391},
  year={2006},
  publisher={EDP Sciences},
  doi={10.1051/m2an:2006013}
}

@article{fumagalli2013NonMaching,
title = {A numerical method for two-phase flow in fractured porous media with non-matching grids},
journal = {Advances in Water Resources},
volume = {62},
pages = {454-464},
year = {2013},
note = {Computational Methods in Geologic CO2 Sequestration},
issn = {0309-1708},
doi = {https://doi.org/10.1016/j.advwatres.2013.04.001},
url = {https://www.sciencedirect.com/science/article/pii/S0309170813000523},
author = {Alessio Fumagalli and Anna Scotti}
}

@article{chen2016adaptive,
  title={Adaptive mixed finite element methods for Darcy flow in fractured porous media},
  author={Chen, Huangxin and Salama, Amgad and Sun, Shuyu},
  journal={Water Resources Research},
  volume={52},
  number={10},
  pages={7851--7868},
  year={2016},
  publisher={Wiley Online Library},
  doi={10.1002/2015WR018450}
}

@article{zhao2022adaptive,
  title={An adaptive discontinuous Galerkin method for the {D}arcy system in fractured porous media},
  author={Zhao, Lina and Chung, Eric},
  journal={Computational Geosciences},
  volume={26},
  number={6},
  pages={1581--1596},
  year={2022},
  publisher={Springer},
  doi={10.1007/s10596-022-10171-5}
}

@article{chen2017residual,
  title={A residual-based a posteriori error estimator for single-phase Darcy flow in fractured porous media},
  author={Chen, Huangxin and Sun, Shuyu},
  journal={Numerische Mathematik},
  volume={136},
  number={3},
  pages={805--839},
  year={2017},
  publisher={Springer},
  doi={10.1007/s00211-016-0851-9}
}

@incollection{gander2009algorithm,
author="Gander, Martin J.
and Japhet, Caroline",
editor="Bercovier, Michel
and Gander, Martin J.
and Kornhuber, Ralf
and Widlund, Olof",
title="An Algorithm for Non-Matching Grid Projections with Linear Complexity",
booktitle="Domain Decomposition Methods in Science and Engineering XVIII",
year="2009",
publisher="Springer Berlin Heidelberg",
address="Berlin, Heidelberg",
pages="185--192",
isbn="978-3-642-02677-5",
doi={10.1007/978-3-642-02677-5\_19}
}

@article{gander2013algorithm,
  title={Algorithm 932: {PANG}: software for nonmatching grid projections in 2{D} and 3{D} with linear complexity},
  author={Gander, Martin J and Japhet, Caroline},
  journal={ACM Transactions on Mathematical Software (TOMS)},
  volume={40},
  number={1},
  pages={1--25},
  year={2013},
  publisher={ACM New York, NY, USA},
  doi={10.1145/2513109.2513115}
}

@article{mccoid2022provably,
  title={A provably robust algorithm for triangle-triangle intersections in floating-point arithmetic},
  author={McCoid, Conor and Gander, Martin J},
  journal={ACM Transactions on Mathematical Software (TOMS)},
  volume={48},
  number={2},
  pages={1--30},
  year={2022},
  publisher={ACM New York, NY},
  doi={10.1145/3513264}
}

@article{smears2020simple,
  title={Simple and robust equilibrated flux a posteriori estimates for singularly perturbed reaction--diffusion problems},
  author={Smears, Iain and Vohral{\'\i}k, Martin},
  journal={ESAIM: Mathematical Modelling and Numerical Analysis},
  volume={54},
  number={6},
  pages={1951--1973},
  year={2020},
  publisher={EDP Sciences},
  doi={10.1051/m2an/2020034}
}

@article{zienkiewicz1992superconvergent,
  title={The superconvergent patch recovery and a posteriori error estimates. {P}art 1: {T}he recovery technique},
  author={Zienkiewicz, Olgierd Cecil and Zhu, Jian Zhong},
  journal={International Journal for Numerical Methods in Engineering},
  volume={33},
  number={7},
  pages={1331--1364},
  year={1992},
  publisher={Wiley Online Library},
  doi={10.1002/nme.1620330702}
}

@article{ainsworth1993posteriori,
  title={A posteriori error estimators for second order elliptic systems. {P}art 2: An optimal order process for calculating self-equilibrating fluxes},
  author={Ainsworth, Mark and Oden, J Tinsley},
  journal={Computers \& Mathematics with Applications},
  volume={26},
  number={9},
  pages={75--87},
  year={1993},
  publisher={Elsevier},
  doi={10.1016/0898-1221(93)90007-I}
}

@article{cai2020robust,
  title={Robust equilibrated a posteriori error estimator for higher order finite element approximations to diffusion problems},
  author={Cai, Difeng and Cai, Zhiqiang and Zhang, Shun},
  journal={Numerische Mathematik},
  volume={144},
  number={1},
  pages={1--21},
  year={2020},
  publisher={Springer},
  doi={10.1007/s00211-019-01075-1}
}

@article{cai2018hybrid,
  title={A hybrid a posteriori error estimator for conforming finite element approximations},
  author={Cai, Difeng and Cai, Zhiqiang},
  journal={Computer Methods in Applied Mechanics and Engineering},
  volume={339},
  pages={320--340},
  year={2018},
  publisher={Elsevier},
  doi={10.1016/j.cma.2018.04.050}
}

@inproceedings{varela2025linear,
  author    = {Varela, Jhabriel and Schaerer, Christian E. and Keilegavlen, Eirik},
  title     = {{A linear potential reconstruction technique based on Raviart--Thomas basis functions for cell-centered finite volume approximations to the Darcy problem}},
  booktitle = {Proceeding Series of the Brazilian Society of Computational and Applied Mathematics},
  volume    = {11},
  number    = {1},
  year      = {2025},
  doi       = {10.5540/03.2025.011.01.0332},
}

@article{wang2019accuracy,
  title={Accuracy analysis of gradient reconstruction on isotropic unstructured meshes and its effects on inviscid flow simulation},
  author={Wang, Nianhua and Li, Ming and Ma, Rong and Zhang, Laiping},
  journal={Advances in aerodynamics},
  volume={1},
  pages={1--31},
  year={2019},
  publisher={Springer},
  doi={10.1186/s42774-019-0020-9}
}

@article{huang2012some,
  title={Some weighted averaging methods for gradient recovery},
  author={Huang, Yunqing and Jiang, Kai and Yi, Nianyu},
  journal={Advances in Applied Mathematics and Mechanics},
  volume={4},
  number={2},
  pages={131--155},
  year={2012},
  publisher={Cambridge University Press},
  doi={10.4208/aamm.10-m1188}
}

@inproceedings{sozer2014gradient,
  title={Gradient calculation methods on arbitrary polyhedral unstructured meshes for cell-centered {CFD} solvers},
  author={Sozer, Emre and Brehm, Christoph and Kiris, Cetin C},
  booktitle={52nd Aerospace Sciences Meeting},
  pages={1440},
  year={2014},
  doi={10.2514/6.2014-1440}
}

@inproceedings{mavriplis2003revisiting,
  title={Revisiting the least-squares procedure for gradient reconstruction on unstructured meshes},
  author={Mavriplis, Dimitri},
  booktitle={16th AIAA computational fluid dynamics conference},
  pages={3986},
  year={2003},
  doi={10.2514/6.2003-3986}
}

@book{brenner2008mathematical,
  title        = {The Mathematical Theory of Finite Element Methods},
  author       = {Brenner, Susanne C. and Scott, L. Ridgway},
  edition      = {3},
  series       = {Texts in Applied Mathematics},
  volume       = {15},
  year         = {2008},
  publisher    = {Springer},
  address      = {New York},
  isbn         = {978-0-387-75933-3}
}

@article{scott1990finite,
  title={Finite element interpolation of nonsmooth functions satisfying boundary conditions},
  author={Scott, L Ridgway and Zhang, Shangyou},
  journal={Mathematics of computation},
  volume={54},
  number={190},
  pages={483--493},
  year={1990},
  doi={10.2307/2008497},
}

@book{mclean2000strongly,
  title        = {Strongly Elliptic Systems and Boundary Integral Equations},
  author       = {McLean, William},
  series       = {Cambridge Studies in Advanced Mathematics},
  volume       = {131},
  year         = {2000},
  publisher    = {Cambridge University Press},
  address      = {Cambridge},
  isbn         = {978-0-521-63218-5}
}

@book{lions2012non,
  title        = {Non-Homogeneous Boundary Value Problems and Applications. Vol. 1},
  author       = {Lions, Jacques-Louis and Magenes, Enrico},
  year         = {2012},
  publisher    = {Springer},
  address      = {Berlin},
  isbn         = {978-3-642-65166-0}
}

@article{geuzaine2009gmsh,
  title={Gmsh: A 3-{D} finite element mesh generator with built-in pre-and post-processing facilities},
  author={Geuzaine, Christophe and Remacle, Jean-Fran{\c{c}}ois},
  journal={International journal for numerical methods in engineering},
  volume={79},
  number={11},
  pages={1309--1331},
  year={2009},
  publisher={Wiley Online Library},
  doi={10.1002/nme.2579},
}

@software{mdnme,
	author = {Jhabriel Varela},
	doi = {10.5281/zenodo.17844410},
	month = dec,
	publisher = {Zenodo},
	swhid = {swh:1:dir:93cc72256b6606506c0cc0bad853a5e6287d4fb2 ;origin=https://doi.org/10.5281/zenodo.17844409;vi sit=swh:1:snp:eec0aa4c5e6fc6e2693c4da150ef87457e51 9175;anchor=swh:1:rel:a6820ff9a77f29ffe95ce5900e73 626e39e5eb85;path=jhabriel- non\_matching\_estimates-5395ca7},
	title = {\texttt{jhabriel/non\_matching\_estimates: {mdnme} {v0.2}}},
	version = {v0.2},
	year = 2025,
}

\end{document}